\documentclass[reqno]{article}
\usepackage{amssymb,amsmath}
\usepackage{amsbsy}
\usepackage{color}
\usepackage[latin1]{inputenc}

\def\hnc      {{\boldsymbol {\mathcal{H}}}}

\def\v        {{\boldsymbol v}}
\def\u        {{\boldsymbol u}}
\def\w        {{\boldsymbol w}}

\def\z        {{\boldsymbol z}}
\def\f        {{\boldsymbol f}}
\def\h        {{\boldsymbol h}}
\def\q        {{\boldsymbol q}}

\def\be       {{\boldsymbol \beta}}

\def\nchi     {{\boldsymbol \chi}}

\def\vph      {{\boldsymbol \varphi}}
\def\nnu      {{\boldsymbol \nu}}
\def\te       {{\boldsymbol \theta}}
\def\kxi      {{\boldsymbol \xi}}
\def\Q        {{\boldsymbol Q}}
\def\L        {{\boldsymbol L}}

\def\H        {{\boldsymbol H}}

\def\Q        {{\boldsymbol Q}}

\newcommand{\disp}{\displaystyle}

\numberwithin{equation}{section}

\newtheorem{theorem}{Theorem}[section]
\newtheorem{proposition}{Proposition}[section]

\newenvironment{proof}{{\bf  Proof}:}{\hfill $\square$\newline}
\newtheorem{definition}{Definition}[section]

\newtheorem{remark}{Remark}[section]
\begin{document}

\title{ Remarks on Hierarchic Control  for a
Linearized Micropolar Fluids System  in Moving Domains}

\author
{\quad Isa\'ias Pereira de Jesus\thanks{ Corresponding author:
Dpto.~Matem\'atica, Universidade Federal do Piau\'i,\,\, Campus\;\;\;  Ministro
Petrônio Portela - Ininga, 64049-550, Teresina, PI, Brasil, {\tt
isaias@ufpi.edu.br}.}}

\date{}
\maketitle

\begin{abstract} We study a Stackelberg strategy subject to the evolutionary linearized micropolar
fluids equations in domains with moving boundaries, considering  a Nash
multi-objective equilibrium (non necessarily cooperative) for the "follower players"
(as is called in the economy field) and an optimal problem for the leader player
with approximate controllability objective. We will obtain the following  main
results : the existence and uniqueness of Nash equilibrium and its
characterization, the approximate controllability of the linearized micropolar
system with respect to the leader control and the existence and uniqueness  of
the Stackelberg-Nash problem, where the optimality system for the leader is
given.
\end{abstract}



\section{Introduction}
In recent years, much attention has been given  to the  investigation of  new
classes of problems in differential equations of hydrodynamics. Control problems
for the Navier-Stokes  equations and other models of fluid mechanics are
examples of these. A number of papers and books (see, for instance \cite {AA},
\cite {J}, \cite {Bli},  \cite {FCG1}, \cite {Fur}, \cite {MD}, \cite {LZ2}, \cite {St} and
references therein) deal with the theoretical and numerical study of the above
mentioned problems.

When considering vectorial or multi-objective problems it is not usual to have the
existence of a simultaneously optimal point. Several concepts of solution have
been introduced in economy and game theory. For example, we could consider
the following:

1.  Nash (non-cooperative ) equilibrium solutions (see \cite {N}),

 2. Stackelberg (hierarchic and cooperative) solutions (see \cite {Sb}),

3. Pareto (strong) solutions (see \cite {P1}).

Although the Pareto concept was the first solution introduced in the literature and
it is based on a vectorial argument, it is often difficult to analyze it
mathematically. On the other hand, the other two concepts, based on scalar
optimization (i.e. relying on the idea of comparing with controls associated to
scalar functionals), have been also considered.

Up to now, the main works on solution of these kinds for (vectorial) optimal
control with PDE constraints are the following :

1. The papers by J.-L. Lions \cite{L10, L11, L12}, where Pareto and Stackelberg
solutions are considered (using a Pareto related with a continuum of objective
functionals, and not a finite number of them). Díaz \cite{D1} and Díaz and  J.-L.
Lions \cite{DL}, have also given some results for Stackelberg-Nash strategies
with linear parabolic PDE constraints. We can also mention  Limaco et al. \cite
{LIM} in which the authors present Stackelberg-Nash equilibrium in non
cylindrical domains.

2. The papers by Ramos et al. \cite{RA1, RA2} , where the authors analyze the
Nash equilibrium for constraints given by linear parabolic and Burges equations
from the mathematical and numerical viewpoints.

In this article we study hierarchic control for two-dimensional incompressible
time-dependent linearized micropolar fluids in a domain of $\mathbb{R}^2_x
\times \mathbb{R}_t$ whose boundary is moving with respect to $t$, for $t \in
[0,T]$ and $T>0$. More precisely, we consider an open bounded domain
$\widehat{Q}$ of $\mathbb{R}^2_x \times \mathbb{R}_t$ which is the union of
open bounded sets $\Omega_t \subset \mathbb{R}^2_x$ where $\Omega_t$ are
deformations of a fixed set $\Omega$ of $\mathbb{R}^2_x$ by a diffeomorphism
$\tau_t$ to be defined below. From now on, we will write $\mathbb{R}^2$ instead
of $\mathbb{R}^2_x$.

Thus, let $\Omega$ be a  fixed, non-empty, open bounded set of
$\mathbb{R}^2$, whose points are represented by $y=(y_1,y_2)$.

Let $\Omega_t$ be the diffeomorphic image of $\Omega$ by the matrix valued
function
$$
\begin{array}{ccc}
[0,T]&\to &\mathbb{R}^{2^2} \\
t &\longmapsto & K(t).
\end{array}
$$

The points of $\Omega_t$ are represented by $x=(x_1,x_2)$. Thus we have
$$x=K(t)y,\;\;\mbox{for}\;\;i=1,2.$$

The non-cylindrical domain $\widehat{Q}$ of $\mathbb{R}^2_x \times
\mathbb{R}_t$ is defined by
$$\widehat{Q}=\displaystyle \bigcup_{0\leq t\leq T} \{\Omega_t \times \{t\}\}.$$

If the boundary of $\Omega_t$ is $\Gamma_t$, then the lateral
boundary of $\widehat{Q}$ is
$$\widehat{\Sigma}=\displaystyle \bigcup_{0\leq t\leq T} \{\Gamma_t \times \{t\}\}.$$

We represent by $Q$ the cylinder $Q=\Omega \times [0,T[$, with
lateral boundary $\Sigma$ given by $\Sigma=\Gamma \times [0,T[$,
where $\Gamma$ is the boundary of $\Omega$.

We assume the following hypothesis on $K(t)$:
$$K(t)=k(t)M,$$
where $k$ is a real function $k:[0,T]\rightarrow\mathbb{R}$, such that $k\in
C^2([0,T]),\,k(t)\geq k_0>0$, for some constant $k_0$  and $M$ is a constant
invertible $2\times 2$ matrix.

Thus we have a natural diffeomorphism $\displaystyle \tau_t:Q \to \widehat{Q}$ defined by
$$(y,t) \in Q \to (x,t) \in
\widehat{Q},\;\;\mbox{where}\;\;x=K(t)(y).$$

We represent by ${\mathcal{O}}_1,{\mathcal{O}}_2 \mbox { and } \mathcal{O}$, non-empty disjoint open subsets of $\Omega$. It means, $\mathcal{O}\cap{\mathcal{O}}_i$, for $i=1,2$, and ${\mathcal{O}}_i\cap{\mathcal{O}}_j$ empties for $i\neq j$.

We denote by ${\mathcal{O}}_t,{\mathcal{O}}_{it},\;i=1,2$, the images of
$\mathcal{O},{\mathcal{O}}_i$, by the mapping $\tau_t$. We suppose the
intersections
${\mathcal{O}}_t\cap{\mathcal{O}}_{it},\;{\mathcal{O}}_{it}\cap{\mathcal{O}}_{jt},\;i\neq
j,$ empties for $1\leq i,j\leq 2 $. In these conditions, we define the non-cylindrical
domains
$$\displaystyle{\widehat{\mathcal{O}}}=\bigcup_{0<t<T}\{{\mathcal{O}}_t\times\{t\}\},\;\;{\widehat{\mathcal{O}}_{1}}=\bigcup_{0<t<T}
\{{\mathcal{O}}_{1t}\times\{t\}
\},{\widehat{\mathcal{O}}_{2}}=\bigcup_{0<t<T}\{{\mathcal{O}}_{2t}\times\{t\}\}$$ which are contained in ${\widehat{Q}}$ and are images by $\tau_t$ of the cylinders of $Q$ with bases on $\mathcal{O}$ and ${\mathcal{O}}_1,{\mathcal{O}}_2$.

This paper was inspired on the ideas contained in the work of Jesus and
Menezes \cite{Isa},  where we investigate a similar question of hierarchic control
employing the Stackelberg-Nash strategy,  where the dynamic of the system is
given by a PDE from fluid mechanics.  The differences of the results obtained in
\cite {Isa} for this paper are in the existence and uniqueness of Nash equilibrium
( see Proposition \ref {pro2.1}) and in the optimality system for the leader ( see
Theorem \ref {theor3}). We assume that we can act in the dynamic of a
linearized micropolar fluid problem in the  velocities and pressure formulation by
a hierarchy of controls. According to the formulation given by H. Von Stackelberg
in \cite{Sb}, there are local controls, called followers and global controls, called
leaders. In fact, several followers are considered with their  own objectives and
the leader has an approximate controllability objective admitting Nash equilibrium
for the followers.

A micropolar fluid is a viscous, non-Newtonian fluid with local microstructure,
which contains suspensions of rigid particles. A mathematical description of
such fluids, which cannot be described, rheologically, by classical Navier-Stokes
systems (especially when the diameter of the domain of flow becomes small),
was first given by Eringen \cite{E}. Animal blood, liquid crystals (with dumbbell
type molecules), polymeric fluids, and certain colloidal fluids, whose fluid
elements exhibit microrotations and complex biological structures, are examples
of fluids modeled by micropolar fluid theory (see Eringen \cite{E2}  and Popel et
al. \cite{P}). Various problems associated with the micropolar fluid model have
been studied recently (see, e.g., Calmelet and Rosenhaus \cite{C}, Lukaszewicz
\cite{L}, Yamaguchi\cite{Y}). For the optimal control problems, see Stavre
\cite{St}, in which the author controls the blood pressure.

We consider the following linearized (around $(\widehat{\h},\widehat{\te})$)
micropolar fluid system
\begin{equation}
\left\{
\begin{array}
[c]{lll}%
\widehat{\z}_{t}-\Delta\widehat{\z}+(\widehat{\h}\cdot\nabla)\widehat{\z} +(\widehat{\z}\cdot\nabla)\widehat{\h}+\nabla
\widehat{p}=\nabla\times \widehat{w}+\widehat{\f}\widehat{\nchi}_{\mathcal{\widehat{O}}}+\widehat{\v}^{(1)}\widehat{\nchi}_{\mathcal{\widehat{O}}_{1}}
+\widehat{\v}^{(2)}\widehat{\nchi}_{\mathcal{\widehat{O}}_{2}} & \mbox{\rm in} & \widehat{Q},\\
\widehat{w}_{t}-\Delta \widehat{w}+\widehat{\h}\cdot\nabla \widehat{w}+\widehat{\z}\cdot\nabla\widehat{\te}
=\nabla\times\widehat{\z}+\widehat{g}\widehat{\chi}_{\mathcal{\widehat{{O}}}}+\widehat{u}^{(1)}\widehat\chi_{\mathcal{\widehat{{O}}%
}_{1}}+\widehat{u}^{(2)}\widehat\chi_{\mathcal{\widehat{{O}}}_{2}} & \mbox{\rm in} & \widehat{Q},\\
\nabla\cdot\widehat{\z}=0 & \mbox{\rm in} & \widehat{Q},\\
\widehat{\z}=0,\quad \widehat{w}=0 & \mbox{\rm on} & \widehat{\Sigma},\\
\widehat{\z}(0)=\widehat{\z}^{0},\quad \widehat{w}(0)=\widehat{w}^{0} & \mbox{\rm in} & \Omega_0,
\end{array}
\right.  \label{2}%
\end{equation}
where $(\widehat{\h},\widehat{\te})$ belongs to the class
\begin{equation}
\left\{
\begin{array}
[c]{l}%
\left(  \widehat{\h},\widehat{\te}\right)  \in \textbf{L}^{\infty}\left(  \widehat{Q}\right)  \times L^{\infty
}\left(  \widehat{Q}\right)  ,\\
\left(  \widehat{\h}_{t},\widehat{\te}_{t}\right)  \in L^{2}\left(  0,T;\textbf{L}^{r}\left(
\Omega_t\right)  \right)  \times L^{2}\left(  0,T;L^{r}\left(  \Omega_t\right)
\right),  \quad\text{for some}\quad r>1.
\end{array}
\right.  \label{en14}%
\end{equation}

Spaces of $\mathbb{R}^2$-valued functions, as well as their elements, are represented by bold-face letters.

In the system (\ref{2}), $\widehat{\z}=\widehat{\z}(x,t)=(\widehat
{\z_{1}}(x,t),\widehat {\z_{2}}(x,t))$ represents the fluid velocity field,
$\widehat{w}=\widehat{w}(x,t)$ is the angular velocity and $\widehat{z}^{0}$ and
$\widehat{w}^{0}$ are,  respectively,   the
velocity and angular velocity  at time $t=0$. The subset $\mathcal{\widehat{O}}%
\subset\Omega_t$ is the control domain (which is supposed to be as small as
wished), $\mathcal{\widehat{{O}}}_{1}$,
$\mathcal{\widehat{{O}}}_{2}\subset\mathcal{\widehat{{O}}}$ the secondary
control domains, the functions $\widehat{\f}$, $\widehat{g}$ are called leaders
controls and $\widehat{\v}^{(i)}$, $\widehat{u}^{(i)}$ $(i=1,2)$ the followers
controls. The quantities $(\widehat{\z}\cdot\nabla)\widehat{\z}$,
$\nabla\times\widehat{\z}$
and $\nabla\times \widehat{w}$ are defined, respectively,  by%
\[
(\widehat{\z}\cdot\nabla)\widehat{\z}=(\widehat{\z}\cdot\nabla \widehat{z}_{1},\widehat{\z} \cdot\nabla \widehat{z}_{2}), \;\;\;\;\
\nabla\times\widehat{\z}=\frac{\partial \widehat{z}_{2}}{\partial x_{1}}-\frac{\partial
\widehat{z}_{1}}{\partial x_{2}}\quad\mbox{\rm and\quad}\nabla\times \widehat{w}=(
\frac{\partial \widehat{w}}{\partial x_{2}},-\frac{\partial \widehat{w}}{\partial
x_{1}}) .
\]
By $\widehat{\nchi}_{\widehat{\mathcal{O}}},\widehat{\nchi}_{\widehat{\mathcal{O}}_{it}},i=1,2,$ we denote the characteristic functions on $\widehat{\mathcal{O}}$ and ${\widehat{\mathcal{O}}}_{it}$.

We assume  that the objective of the leader is of controllability type. In fact, the
leader wants to drive the state velocity $(\widehat{\z}, \widehat{w})$ at the final
time T as close as possible to a wished state $ (\widehat{\z}^T, \widehat{w}^T)$,
without a big cost for the pair of controls $(\widehat{\f}, \widehat{g})$. On the
other hand, the main objective of the followers acting on the pair of controls
$(\widehat{\v}^{\left( i\right) }, \widehat{u}^{\left(  i\right)  })$ $(i=1,2)$, is to hold
the state $(\widehat{\z}, \widehat{w})$ near to a desired state $(
\widehat{\z}_{i,d}, \widehat{w}_{i,d})$ for all time $ t \in (0,T)$ in an observability
region ${\mathcal{\widehat{O}}}_{i,d}$, without a big cost for the pair of controls
$(\widehat{\v}^{\left(  i\right)  }, \widehat{u}^{\left( i\right)  })$.

Concerning to control results for micropolar fluids, we cite the work by
Fern\'andez-Cara and Guerrero \cite{FCG}, where they studied the local exact
controllability of trajectories, and \cite{A}, where it was obtained a global exact
controllability for Galerkin's approximations of this system.  We mention also
Araruna et al.  \cite {A1} in which the authors present the approximate
controllability of Stackelberg-Nash strategies for linearized micropolar fluids.

In this paper, we present  the following  results : the existence and uniqueness of
Nash equilibrium and its characterization, the approximate controllability of the
linearized micropolar system with respect to the leader control and the existence
and uniqueness  of the Stachelberg-Nash problem, where the optimality system
for the leader is given.

The rest of the paper is organized as follows. In Section \ref{Sec2.0} we present
the problem.  Section \ref{Sec2} is devoted to establish existence, uniqueness
and characterization  of the Nash equilibrium. In  Section \ref{Sec3} we study the
approximate controllability with respect to the leader control. In Section
\ref{Sec4} we give the optimality system for the leader.  Finally, in Section
\ref{sec5} we present the conclusions.

\section{Problem formulation}\label{Sec2.0} The following vector spaces, usual in the context of incompressible
fluids, will be used along the paper:
\[
\mathbf{V}_t=\left\{\vph\in\mathbf{H}_{0}^{1}(\Omega_t
);\;\nabla\cdot\vph=0\;\mbox{\rm in}\;\Omega_t\right\}
\]
and%
\[
\mathbf{H}_t=\left\{
\vph\in\mathbf{L}^{2}(\Omega_t);\;\nabla
\cdot\vph=0\;\mbox{\rm
in}\;\Omega_t,\;\vph\cdot\mathbf{\nnu }=0\;\mbox{\rm
on}\;\Gamma_t\right\}  .
\]%

We will denote by $\displaystyle \mathbf{\nnu }$ the outward unit normal to
$\Omega_t$ at the point $ x \in \Gamma_t$. Let $\mathcal{\widehat{{O}}}_{1,d}$,
$\mathcal{\widehat{{O}}}_{2,d}$ be open subsets of $\Omega_t$, representing
the observability domains for the  followers, which are localized in an arbitrary
manner in $\Omega_t$.

{\bf Cost functionals in non-cylindrical $\widehat{\Q}.$\;\;}Associated to the
solution $(\widehat{\z},\widehat{w})$ of (\ref{2}), we will consider the following
(secondary) functionals
\begin{equation}
\widehat{J}_{i}\left(\widehat{\f}, \widehat{g};\widehat{\v},\widehat{u}\right)  =\frac{\widehat{\alpha}_{i}}{2}\int_{0}%
^{T}\int_{\mathcal{\widehat{O}}_{i,d}}\left\vert \widehat{\z}-\widehat{\z}_{i,d}\right\vert
^{2}dxdt+\frac{\widehat{\mu}_{i}}{2}\int_{0}^{T}\int_{\mathcal{\widehat{O}}_{i}}\vert
\widehat{\v}^{(i)}\vert ^{2}dxdt,\quad i=1,2, \label{3}%
\end{equation}%
\begin{equation}
\widetilde{\widehat{J}_{i}}\left(\widehat{\f}, \widehat{g};\widehat{\v},\widehat{u}\right)  =\frac{\widetilde{\widehat{\alpha}_{i}}}%
{2}\int_{0}^{T}\int_{\mathcal{\widehat{O}}_{i,d}}\left\vert
\widehat{w}-\widehat{w}_{i,d}\right\vert
^{2}dxdt+\frac{\widetilde{\widehat{\mu}_{i}}}{2}\int_{0}^{T}\int_{\mathcal{\widehat{O}}_{i}%
}\vert \widehat{u}^{(i)}\vert ^{2}dxdt,\quad i=1,2, \label{4}%
\end{equation}
and the (main) functional
\begin{equation}
\widehat{J}\left(\widehat{\f}, \widehat{g}\right)  =\frac{1}{2}\int_{0}^{T}\int_{\mathcal{\widehat{O}}}
\big(|\widehat{\f}|^2 + |\widehat{g}|^2\big)  dxdt, \label{5}%
\end{equation}
where $\widehat{\v}=(\widehat{\v}^{(1)},\widehat{\v}^{(2)}) $, $\widehat{u}=(\widehat{u}^{(1)},\widehat{u}^{(2)}) $,
$\widehat{\alpha_{i}},\widetilde{\widehat{\alpha}}_{i}>0$, $\widehat{\mu}_{i},\widetilde
{\widehat{\mu}}_{i}>0$ are constants, and $\widehat{\z}_{i,d},$ $\widehat{w}_{i,d}$ are given
functions in $L^2(\mathcal{\widehat{O}}_{i,d} \times (0,T)).$

\begin{remark}\label{bdfnc} From the regularity and uniqueness of the solution $\disp (\widehat{\z}, \widehat{w})$ of (\ref{2}) (see Remark \ref{rsol} ) the cost functionals $\widehat{{J}}_{i}$, $\widetilde{\widehat{J}_{i}}$ and $\widehat{{J}}$ are well defined.
\end{remark}

The followers $\widehat{\v}^{\left(  i\right)  },$ $\widehat{u}^{\left(  i\right)  }$ $(i=1,2)$
assume that the leaders have made choices $\widehat{\f}$ and $\widehat{g}$ and they, respectively,
try to find a Nash equilibrium of their cost $\widehat{J}_{i}$ and $\widetilde{\widehat{J}_{i}}$
$(i=1,2),$ i.e. they look for controls $\widehat{\v}^{\left(  i\right)  },$ $\widehat{u}^{\left(
i\right)  }$ $(i=1,2)$, depending on $\widehat{\f}$ and $\widehat{g}$, such that%
\begin{equation}
\widehat{J}_{1}(  \widehat{\f},\widehat{g};\widehat{\v},\widehat{u})  =\text{~}\underset{\widetilde{\widehat{\v}}^{\left(  1\right)  }}%
{\min}\text{~}\widehat{J}_{1}(  \widehat{\f},\widehat{g};\widetilde{\widehat{\v}}^{\left(  1\right)  },\widehat{\v}^{\left(  2\right)
},\widehat{u})  ,\quad \widehat{J}_{2}\left(  \widehat{\f},\widehat{g};\widehat{v},\widehat{u}\right)  =\text{~}\underset{\widetilde{\widehat
{\v}}^{\left(  2\right)  }}{\min}\text{~}\widehat{J}_{2}(  \widehat{\f},\widehat{g};\widehat{v}^{\left(  1\right)
},\widetilde{\widehat{\v}}^{\left(  2\right)  },\widehat{u})  \label{en6.1}%
\end{equation}
and%
\begin{equation}
\widetilde{\widehat{J}_{1}}\left(  \widehat{\f},\widehat{g};\widehat{\v},\widehat{u}\right)  =\text{~}\underset{\widetilde{\widehat{u}}^{\left(
1\right)  }}{\min}\text{~}\widehat{J}_{1}(  \widehat{\f},\widehat{g};\widehat{\v},\widetilde{\widehat{u}}^{\left(  1\right)
},\widehat{u}^{\left(  2\right)  })  ,\quad \widetilde{\widehat{J}}_{2}\left(  \widehat{\f},\widehat{g};\widehat{\v},\widehat{u}\right)
=\text{~}\underset{\widetilde{\widehat{u}}^{\left(  2\right)  }}{\min}\text{~}\widehat{J}_{2}(
\widehat{\f},\widehat{g};\widehat{\v},\widehat{u}^{\left(  1\right)  },\widetilde{\widehat{u}}^{\left(  2\right)  }).
\label{en7.1}%
\end{equation}
The pair $(\widehat{\v},\widehat{u})$ which satisfies (\ref{en6.1}) and (\ref{en7.1}) is called Nash equilibrium.

Notice that, since, for $(i=1,2)$,  the functionals $\widehat{J}_{i}$ and
$\widetilde{\widehat{{J}_{i}}}$ are convex, then $(\widehat{\v},\widehat{u})$ is a
Nash equilibrium if and only if
\begin{equation}
\displaystyle D_{\widetilde{\widehat{\v}}^{(i)}}\widehat{J}_{i}(  \widehat{\f},\widehat{g};\widehat{\v},\widehat{u}) = 0, \;\;\;\forall \widetilde{\widehat{\v}}^{(i)}  \in \mathbf{L}^{2}(  \mathcal{\widehat{O}}_{i}\times(0,T))
\label{en10}%
\end{equation}
and
\begin{equation}
\displaystyle D_{\widetilde{\widehat{\u}}^{(i)}}\widetilde{\widehat{J}_{i}}(  \widehat{\f},\widehat{g};\widehat{\v},\widehat{u}) = 0,  \;\;\;\forall \widetilde{\widehat{\u}}^{(i)}  \in {L}^{2}(  \mathcal{\widehat{O}}_{i}\times(0,T)),
\label{en11}
\end{equation}
where $\displaystyle D_{\widetilde{\widehat{\v}}^{(i)}}\widehat{J}_{i}$ (respec.
$\displaystyle D_{\widetilde{\widehat{\u}}^{(i)}}\widetilde{\widehat{J}_{i}})$  is a
Gateaux derivative of $\widehat{J_{i}}$ (respec. $\widetilde{\widehat{J}_{i}})$ with
respect to ${\widetilde{\widehat{\v}}^{(i)}}$ (respec.
${\widetilde{\widehat{\u}}^{(i)}}).$

After  finding the Nash equilibrium for each $(\widehat{\f},\widehat{g})$, we  look
for any optimal control $(\overline{\widehat{\f}},\overline{\widehat{g}})$ such that
\begin{equation}
J(  \overline{\widehat{\f}},\overline{\widehat{g}})  =\text{~}\underset{(
\widehat{\f},\widehat{g})  }{\min}\text{~}J\left(  \widehat{\f},\widehat{g}\right),  \label{en8}%
\end{equation}
subject to the restriction of the approximate controllability type%
\begin{equation}
(\widehat{\z}(\cdot,T,\widehat{\f},\widehat{g},\widehat{\v},\widehat{u}), \widehat{w}(\cdot,T,\widehat{\f},\widehat{g},\widehat{\v},\widehat{u}))\in B_{\mathbf{L}^{2}(\Omega_t)}(\widehat{\z}^{T}%
,\epsilon)\times B_{L^{2}(\Omega_t)}(\widehat{w}^{T},\epsilon),\label{en9}%
\end{equation}
where $B_X(C, \epsilon)$ denotes the ball in $X$ with center $C$ and radius
$\epsilon$.

The methodology, in the present paper, consists in to change the non-cylindrical
state equation (\ref{2}) to a cylindrical one (\ref{eq10}) by the diffeomorphism
$\tau_t$. Its inverse is given by
$$\tau_t^{-1}: \widehat{Q} \to Q \mbox{ such that } (x,t) \in \widehat Q \to (y,t) \in Q,\;\;\mbox{with}\;\;y=K^{-1}(t)(x).$$

We set
\begin{equation}\label{mva}\left|
\begin{array}{l}
\widehat{\z}(x,t)=\z(K^{-1}(t)x,t)\\[3pt]\displaystyle
\widehat{w}(x,t)=w(K^{-1}(t)x,t)\\[3pt]\displaystyle
\widehat{p}(x,t)=p(K^{-1}(t)x,t).
\end{array}\right.
\end{equation}

In view of (\ref{mva}) and employing the notation
$$K(t)=k(t)M=(\alpha_{ij}(t)),\;\;K^{-1}(t)=\displaystyle\frac{1}{k(t)}M^{-1}=(\beta_{ij}(t)).$$
$$x=K(t)y,\;\;y=K^{-1}(t)x$$
 $$ \displaystyle
x_r=\sum^{2}_{j=1}\alpha_{rj}(t)y_j,\;\;\displaystyle
y_l=\sum^{2}_{r=1}\beta_{lr}(t)x_r,
$$
we obtain the following identities
$$\displaystyle\frac{\partial y_l}{\partial
t}=\sum^2_{r,j=1}\beta'_{lr}(t)\alpha_{rj}(t)y_j\mbox{  and
}\displaystyle\frac{\partial y_l}{\partial x_j}=\beta_{lj}(t).$$

Then
 \begin{equation}\begin{array}{l}
\displaystyle\frac{\partial \widehat{\z}_i}{\partial
t}(x,t)=\sum^{2}_{j,l,r=1}\beta'_{lr}(t)\alpha_{rj}(t)y_j\frac{\partial
\z_i}{\partial y_l}(y,t)+\frac{\partial \z_i}{\partial t}(y,t).
\label{eqA1}
\end{array}
\end{equation}

By (\ref{eqA1}) we have
\begin{equation}\begin{array}{l}\displaystyle\frac{\partial \widehat{\z}}{\partial
t}(x,t)=-yK'(t)K^{-1}(t)\nabla \z+\z_{t},
\label{eqA2}
\end{array}
\end{equation}
because $K(t).(K'(t))'=-k'(t)k^{-1}(t)$. As $\displaystyle\frac{\partial y_l}{\partial x_j}=\beta_{lj}(t)$,
we obtain
\begin{equation}
\displaystyle\frac{\partial \widehat{\z}_i}{\partial
x_j}(x,t)=\sum^{2}_{l=1}\beta_{lj}(t)\frac{\partial \z_i}{\partial
y_l}(y,t), \label{eqA3}
\end{equation}
that implies
$$\displaystyle\frac{\partial^2 \widehat{\z}_i}{\partial x_j^2}(x,t)=\sum^{2}_{l,r=1}\beta_{lj}(t)
\beta_{rj}(t)\frac{\partial^2 \z_i}{{\partial y_l}{\partial
y_r}}(y,t).$$

Thus,
\begin{equation}
\displaystyle\Delta
\widehat{\z}_i(x,t)=\sum^{2}_{j,l,r=1}\beta_{lj}(t)
\beta_{rj}(t)\frac{\partial^2
\z_i}{{\partial y_l}{\partial y_r}}(y,t), \label{eqA4}
\end{equation}
$$\displaystyle\nabla\cdot\widehat{\z}=\frac{1}{k(t)}\nabla\cdot(M^{-1}\z^T)$$

Furthermore,
$$\left(\widehat{\h}.\nabla\right)\widehat{\z}=(\h(K^{-1}(t))^T)\cdot\nabla)\z,\;\;\left(\widehat{\z}\cdot\nabla\right)\widehat{\h}=(\z(K^{-1}(t))^T).\nabla)\h$$
$$\nabla\widehat{p}=\nabla p(y,t)K^{-1}(t),\;\;
\nabla\times\widehat{\w}=\overline{K^{-1}(t)}\nabla\times \w,$$ where $\overline{K^{-1}(t)}$ is the cofactor matrix of $K^{-1}(t).$

Also,
$$\widehat{\f}\nchi_{\widehat{\mathcal{O}}}=\f\nchi_{\mathcal{O}\times[0,T]},\;\;\widehat{\v}^{(1)}\nchi_{\widehat{\mathcal{O}}_1}=\v^{(1)}\nchi_{\mathcal{O}_1\times[0,T]}$$
$$\displaystyle\frac{\partial \widehat{w}}{\partial
t}(x,t)=-yK'(t)K^{-1}(t)\nabla w+w_{t},\;\;\Delta
\widehat{w}(x,t)=\sum^{2}_{j,l,r=1}\beta_{lj}(t)
\beta_{rj}(t)\frac{\partial^2
w}{{\partial y_l}{\partial y_r}}(y,t),$$
$$\widehat{\h}\cdot\nabla\widehat{w}=\h(K^{-1}(t))^T\cdot\nabla w,\;\;\widehat{w}\cdot\nabla\widehat{\te
}=z(K^{-1}(t))^T\cdot\nabla\te,$$
$$\nabla\times\widehat{\z}=\displaystyle\sum_{i=1}^2\beta_{ii}\nabla\times \z+\sum_{i,j=1}^2(-1)^{j+1}\beta_{ij}\frac{\partial \z_i}{\partial y_{3-j}}.$$

Therefore the change of variables $(\ref{mva})$, transforms the the
initial-boundary value problem (\ref{2}) into the equivalent system
\begin{equation}
\displaystyle\left|\begin{array}{l}
\displaystyle \z_t-yK'(t)K^{-1}(t)\nabla \z-\sum^{2}_{j,l,r=1}\beta_{lj}(t)
\beta_{rj}(t)\frac{\partial^2
\z}{{\partial y_l}{\partial y_r}}(y,t)+((\h(K^{-1}(t))^T)\cdot\nabla)\z
\\\displaystyle+(\z(K^{-1}(t))^T)\cdot\nabla)\h+\nabla p(y,t)K^{-1}(t)=\overline{K^{-1}(t)}\nabla\times w+\f\nchi_{\mathcal{O}\times[0,T]}\\+\v^{(1)}\nchi_{\mathcal{O}_1\times[0,T]}+\v^{(2)}\nchi_{\mathcal{O}_2\times[0,T]}
\hspace*{6,7cm}\mbox { in } Q
\\\displaystyle w_t-yK'(t)K^{-1}(t)\nabla w-\sum^{2}_{j,l,r=1}\beta_{lj}(t)
\beta_{rj}(t)\frac{\partial^2
w}{{\partial y_l}{\partial y_r}}(y,t)+\h(K^{-1}(t))^T\cdot\nabla w
\\+\z(K^{-1}(t))^T\cdot\nabla\te=\displaystyle\sum_{i=1}^2\beta_{ii}\nabla\times \z+\sum_{i,j=1}^2(-1)^{j+1}\beta_{ij}\frac{\partial \z_i}{\partial y_{3-j}}
\\+g\chi_{\mathcal{O}\times[0,T]}+\displaystyle u^{(1)}\chi_{\mathcal{O}_1\times[0,T]}+u^{(2)}\chi_{\mathcal{O}_2\times[0,T]}\hspace*{4.9cm}\mbox{ in } Q
\\\nabla\cdot(M^{-1}\z^T)=0\hspace*{8,8cm}\mbox{ in }Q
\\\z=0,\;\;w=0\hspace*{9.3cm}\mbox{ on } \Sigma
\\\z(0)= z^{0},\;\;\;\;w(0)=w^{0}\hspace*{7,85cm} \mbox{ in }\Omega\label{eq10}\end{array}\right.
\end{equation}
where $(\h,\te)$ belongs to the class
\begin{equation}
\left\{
\begin{array}
[c]{l}%
\left(  \h,\te\right)  \in \textbf{L}^{\infty}\left(  Q\right)  \times L^{\infty
}\left(  Q\right)  ,\\
\left(  \h_{t},\te_{t}\right)  \in L^{2}\left(  0,T;\textbf{L}^{r}\left(
\Omega\right)  \right)  \times L^{2}\left(  0,T;L^{r}\left(  \Omega\right)
\right)  \quad\text{for some}\quad r>1.
\end{array}
\right.  \label{en14}%
\end{equation}

\begin{remark}\label{rsol} Under the  hypothesis on $K(t)$, we conclude that the bilinear forms
$$\displaystyle \boldsymbol{\alpha}(\boldsymbol z, \boldsymbol \varrho)= \sum^{2}_{j,l,r=1}\beta_{lj}(t)\beta_{rj}(t)
\int_\Omega \frac{\partial\z}{{\partial y_r}}\frac{\partial \boldsymbol \varrho} {\partial y_l}(y,t)\;dy,$$
$$\displaystyle \alpha( z, \varrho)= \sum^{2}_{j,l,r=1}\beta_{lj}(t)\beta_{rj}(t)
\int_\Omega \frac{\partial z}{{\partial y_r}}\frac{\partial \varrho} {\partial
y_l}(y,t)\;dy$$ are bounded and coercive in $\boldsymbol H_0^1(\Omega) \times
\boldsymbol H_0^1(\Omega)$ and in $H_0^1(\Omega) \times H_0^1(\Omega)$,
respectively. By using a similar technique as  in Belmiloudi \cite{AB} (Chapter
12), we can prove the following:

For any $\disp \left(\f, \v^{(1)}, \v^{(2)}, g, u^{(1)}, u^{(2)}\right) \in
\left(\boldsymbol{L}^2(Q)\right)^3\times \left({L}^2(Q)\right)^3$ and any $\disp
(\z^0, w^0) \in \boldsymbol{H} \times L^2(\Omega)$, the system (\ref{eq10})
posses one solution $(\z,w)$ in the class
$$\disp (\z,w)\in \left( \boldsymbol{L}^\infty(0, T; \boldsymbol{H})\cap \boldsymbol{L}^2(0, T; \boldsymbol{V})   \right) \times \left( {L}^\infty(0, T; {H})\cap {L}^2(0, T; {V})   \right),$$

$\mathbf{V}=\left\{\vph\in\mathbf{H}_{0}^{1}(\Omega
);\;\nabla\cdot\vph=0\;\mbox{\rm in}\;\Omega\right\},\;
\mathbf{H}=\left\{
\vph\in\mathbf{L}^{2}(\Omega);\;\nabla
\cdot\vph=0\;\mbox{\rm
in}\;\Omega,\;\vph\cdot\mathbf{\nnu }=0\;\mbox{\rm
on}\;\Gamma\right\}  .
$

Using the diffeomorphism $(y,t) \to (x,t)$, from $Q$ to $\widehat{Q}$, we obtain a
unique solution $(\widehat{\z}, \widehat{w})$ for the problem (\ref{2}) with the
regularity
$$\disp (\widehat{\z}, \widehat{w})\in \left( \boldsymbol{L}^\infty(0, T; \boldsymbol{H}_t)\cap \boldsymbol{L}^2(0, T; \boldsymbol{V}_t)   \right) \times \left( {L}^\infty(0, T; {H}_t)\cap {L}^2(0, T; {V}_t)   \right).$$

\end{remark}

{\bf Cost functionals in  $\Q$.\;\;} From the diffeomorphism $\tau_t^{-1}$ which transform $\widehat{Q}$ in $Q$, we transform the cost functionals $\widehat{J}_{i}$, $\widetilde{\widehat{J}_{i}}$ and $\widehat{J}$ in the cost functionals ${J}_{i}$, $\widetilde{{J}_{i}}$ and ${J}$ defined by

\begin{equation}
{J}_{i}\left({\f}, {g};{\v},{u}\right)  =\frac{{\alpha}_{i}}{2}\int_{0}%
^{T}\int_{\mathcal{{O}}_{i,d}}|\det K(t)|\left\vert {\z}-{\z}_{i,d}\right\vert
^{2}dydt+\frac{{\mu}_{i}}{2}\int_{0}^{T}\int_{\mathcal{{O}}_{i}}|\det K(t)|\vert
{\v}^{(i)}\vert ^{2}dydt, \label{3c}%
\end{equation}%

\begin{equation}
\widetilde{{J}_{i}}\left({\f}, {g};{\v},{u}\right)  =\frac{\widetilde{{\alpha}_{i}}}%
{2}\int_{0}^{T}\int_{\mathcal{{O}}_{i,d}}|\det K(t)|\left\vert
{w}-{w}_{i,d}\right\vert
^{2}dydt+\frac{\widetilde{{\mu}_{i}}}{2}\int_{0}^{T}\int_{\mathcal{{O}}_{i}%
}|\det K(t)|\vert {u}^{(i)}\vert ^{2}dydt \label{4c}%
\end{equation}

\begin{equation}
{J}\left({\f}, {g}\right)  =\frac{1}{2}\int_{0}^{T}\int_{\mathcal{{O}}}|\det K(t)|
(\left\vert {\f} \right\vert^{2}+\left\vert {g}\right\vert^{2})  dydt, \label{5}%
\end{equation}
where $\displaystyle \z_{i,d}$, $w_{i,d}$ are the images of $\widehat{\z}_{i,d}$
and $\widehat{w}_{i,d}$ by the diffeomorphism $\tau_t^{-1}$,
${\v}=({\v}^{(1)},{\v}^{(2)}) $, ${u}=({u}^{(1)},{u}^{(2)}) $, $\mathcal{{{O}}}_{1,d}$,
$\mathcal{{{O}}}_{2,d}$ are the images of $\mathcal{\widehat{{O}}}_{1,d}$,
$\mathcal{\widehat{{O}}}_{2,d}$ via $\tau_t^{-1}$  and
${\alpha_{i}},\widetilde{{\alpha}}_{i}, {\mu}_{i},\widetilde {{\mu}}_{i}>0$ are positive
constants.

\begin{remark}\label{bdf} From the regularity and uniqueness of the solution $\disp (\z, w)$ of (\ref{eq10}) the cost functionals ${J}_{i}$, $\widetilde{{J}_{i}}$ and ${J}$ are well defined.
\end{remark}

The Nash equilibrium for the costs ${J}_{i}$ and $\widetilde{{J}_{i}}$ $(i=1,2)$ is
the pair $({\v},{u})$, which depends on ${\f}$ and ${g}$, solution of
\begin{equation}\left\{
\begin{array}{l}
\disp {J}_{1}(  {\f},{g};{\v},{u})  =\text{~}\underset{\overline{{\v}}^{\left(  1\right)  }}%
{\min}\text{~}{J}_{1}(  {\f},{g};\overline{{\v}}^{\left(  1\right)  },{\v}^{\left(  2\right)
},{u}),\\[7pt]\disp  {J}_{2}\left(  {\f},{g};{v},{u}\right)  =\text{~}\underset{\overline{
{\v}}^{\left(  2\right)  }}{\min}\text{~}{J}_{2}(  {\f},{g};{v}^{\left(  1\right)
},\overline{{\v}}^{\left(  2\right)  },{u})  \label{en6}%
\end{array}\right.
\end{equation}
and%
\begin{equation}
\left\{
\begin{array}{l}
\disp\widetilde{{J}_{1}}\left(  {\f},{g};{\v},{u}\right)  =\text{~}\underset{\overline{{u}}^{\left(
1\right)  }}{\min}\text{~}{J}_{1}(  {\f},{g};{\v},\overline{{u}}^{\left(  1\right)
},{u}^{\left(  2\right)  }), \\[7pt]\disp\widetilde{{J}}_{2}\left(  {\f},{g};{\v},{u}\right)
=\text{~}\underset{\overline{{u}}^{\left(  2\right)  }}{\min}\text{~}{J}_{2}(
{\f},{g};{\v},{u}^{\left(  1\right)  },\overline{{u}}^{\left(  2\right)  }).
\label{en7}%
\end{array}\right.
\end{equation}

\begin{remark}\label{rfund} Since  the functionals ${J}_{i}$ and $\widetilde{{{J}_{i}}}$ $(i=1,2)$ are convex and lower semi-continuous, then $({\v},{u})$ is a Nash equilibrium for the cost functionals
${J}_{i}$ and $\widetilde{{{J}_{i}}}$ if, and only if, it verifies the Euler-Lagrange equation (see Section \ref{Sec2}), that is,
\begin{equation}
\displaystyle D_{\overline{{\v}}^{(i)}}  {J}_{i} (  {\f},{g};{\v},{u}) = 0,  \;\;\;\forall {\overline{\v}}^{(i)}  \in \mathbf{L}^{2}(  \mathcal{{O}}_{i}\times(0,T))
\label{en10}%
\end{equation}
and
\begin{equation}
\displaystyle D_{\overline{{\u}}^{(i)}} \widetilde{ {J}_{i}} (  {\f},{g};{\v},{u}) = 0, \;\;\;\forall {\overline{u}}^{(i)}  \in {L}^{2}(  \mathcal{{O}}_{i}\times(0,T)),
\label{en11}
\end{equation}
where $ D_{\overline{{\v}}^{(i)}}  {J}_{i}$ (respec. $\displaystyle
D_{\overline{{\u}}^{(i)}} \widetilde{ {J}_{i}})$ is a Gateaux derivative of ${J_{i}}$
(respec. $\widetilde{{J}_{i}})$ with respect to $\overline{{\v}}^{(i)}$ (respec.
$\overline{{u}}^{(i)}).$
\end{remark}

Note that if $({\v},{u})$ is the unique Nash equilibrium, depending on ${\f}$ and
${g}$, for (\ref{3c}) and (\ref{4c}), then its transformation, by  $\tau_t^{-1}$, is the
unique Nash equilibrium $(\widehat{\v},\widehat{u})$ for (\ref{3}) and (\ref{4}),
which depends on $\widehat{\f}$ and $\widehat{g}$.


We recall that our initial problem was the controls $\widehat{\f}$, $\widehat{g}$,
$\widehat{\v}^{(1)}$, $\widehat{\v}^{(2)}$, $\widehat{u}^{(1)}$, $\widehat{u}^{(2)}$
work so that the functions $(\widehat{\z},\widehat{w})$, unique solution of
(\ref{2}), reaches   a fixed state $\disp (\widehat{\z}^{T}, \widehat{w}^{T}) \in
\mathbf{L}^{2}(\Omega_t)\times L^2(\Omega_t)$, at time $T$, with cost
functionals defined by (\ref{3}) and (\ref{4}).

From the diffeomorphism $\tau_t$, this problem in $\disp \widehat{Q}$ was
transformed into one equivalent in the cylinder $\disp Q$. Thus, for fixed $\disp
({\z}^{T}, {w}^{T}) \in \mathbf{L}^{2}(\Omega)\times L^2(\Omega)$ the controls
${\f}$, ${g}$, $\v^{(1)}$, $\v^{(2)}$, $u^{(1)}$, $u^{(2)}$ must work so that the
unique solution $(\z, w)$ of (\ref{eq10}), evaluated at $t=T$, reaches the ideal
state $\disp ({\z}^{T}, {w}^{T})$. This will be done in the sense of an approximate
controllability. In fact, it is sufficient to prove that if $\v^{(1)}$, $\v^{(2)}$, $u^{(1)}$,
$u^{(2)}$, depending of $\f$ and $g$, is the unique Nash equilibrium for the cost
functionals (\ref{3c}) and (\ref{4c}), then we have an approximate controllability.
This means that if there exists the Nash equilibrium and $(\z,w)$ is the unique
solution of (\ref{eq10}), then the set generated by
$({\z}(\cdot,T,{\f},{g},{\v},{u}),{w}(\cdot,T,{\f},{g},{\v},{u}))$ is dense in $\disp
\mathbf{L}^2(\Omega) \times L^2(\Omega)$, that is, approximate $\disp ({\z}^{T},
{w}^{T})$. This will be proved in the Section \ref{Sec3}, that is, after  finding the
Nash equilibrium (in Section \ref{Sec2}) for each $({\f},{g})$, we will look for any
optimal control $(\overline{{\f}},\overline{{g}})$ such that
\begin{equation}
J(  \overline{{\f}},\overline{{g}})  =\text{~}\underset{(
{\f},{g})  }{\min}\text{~}J\left(  {\f},{g}\right),  \label{en8}%
\end{equation}
subject to the restriction of the approximate controllability type%
\begin{equation}
({\z}(\cdot,T,{\f},{g},{\v},{u}),{w}(\cdot,T,{\f},{g},{\v},{u}))\in B_{\mathbf{L}^{2}(\Omega)}({\z}^{T}%
,\epsilon)\times B_{L^{2}(\Omega)}({w}^{T},\epsilon),\label{en9}%
\end{equation}
where $B_{X}(C,\epsilon)$ denotes the ball in $X$ with center $C$ and radius
$\epsilon$.

\section{Nash equilibrium}\label{Sec2}

In this section, we determine the existence and uniqueness of a Nash
equilibrium in the sense of (\ref{en6}), (\ref{en7}) and a characterization of this
Nash equilibrium in terms of a solution of an adjoint system.

\subsection{Existence and uniqueness of Nash equilibrium}

For $i = {1,2}$, we  define the spaces
\[
\hnc_{i}=\mathbf{L}^{2}\left(  \mathcal{O}_{i}\times\left(
0,T\right)  \right),\quad\hnc=\hnc_{1} \times \hnc_{2}, \quad \mathcal{H}_{i}=L^{2}\left(  \mathcal{O}_{i}\times\left(
0,T\right)  \right)\quad\text{and}\quad\mathcal{H}=\mathcal{H}_{1} \times \mathcal{H}_{2},
\]
and  consider the operator
\begin{equation*}
\mathbb{L}_{i}\in\mathcal{L}\left(  \hnc_{i} \times \mathcal{H}_{i} ;\mathbf{L}^{2}\left(  Q\right) \times {L}^{2}\left(  Q\right)  \right) ; \;\; \mathbb{L}_{i}(\v^{(i)},u^{(i)}) = ( L_{1,i}\v^{(i)} + \widetilde{L}_{1,i}u^{(i)}, L_{2,i}\v^{(i)}
+ \widetilde{L}_{2,i}u^{(i)}),
\end{equation*}
where%
\begin{equation}
{L}_{1,i}\in\mathcal{L}\left(  \hnc_{i} ; \mathbf{L}^{2}\left(  Q\right)  \right), \;\; {L}_{2,i}\in\mathcal{L}\left(  \hnc_{i} ;L^{2}\left(  Q\right)  \right), \;\; \widetilde{L}_{1,i}\in\mathcal{L}\left( \mathcal{H}_{i} ;
\mathbf{L}^{2}\left(  Q\right),  \right)\;\;
\widetilde{L}_{2,i}\in\mathcal{L}\left(  \mathcal{H}_{i} ;L^{2}\left(  Q\right)  \right)\label{en44}
\end{equation}
and
\begin{equation}\label{en55}
( L_{1,i}\v^{(i)} + \widetilde{L}_{1,i}u^{(i)}, L_{2,i}\v^{(i)} + \widetilde{L}_{2,i}u^{(i)}) = ( \z^{(i)}, w^{(i)}),
\end{equation}
where $\left(  \z^{(i)}, w^{(i)}, p^{(i)}\right)  $ is a solution of following system:
\begin{equation}
\displaystyle\left|\begin{array}{l}
\displaystyle \z^{(i)}_t-yK'(t)K^{-1}(t)\nabla \z^{(i)}-\sum^{2}_{j,l,r=1}\beta_{lj}(t)
\beta_{rj}(t)\frac{\partial^2
\z^{(i)}}{{\partial y_l}{\partial y_r}}(y,t)+(\h(K^{-1}(t))^T)\cdot\nabla)\z^{(i)}
\\\displaystyle+(\z^{(i)}(K^{-1}(t))^T)\cdot\nabla)\h+\nabla p^{(i)}(y,t)K^{-1}(t)=\overline{K^{-1}(t)}\nabla\times w^{(i)}+\v^{(i)}\nchi_{\mathcal{O}_i\times[0,T]}\;\; \mbox{ in } Q
\\\displaystyle w^{(i)}_t-yK'(t)K^{-1}(t)\nabla w^{(i)}-\sum^{2}_{j,l,r=1}\beta_{lj}(t)
\beta_{rj}(t)\frac{\partial^2
w^{(i)}}{{\partial y_l}{\partial y_r}}(y,t)+\h(K^{-1}(t))^T\cdot\nabla w^{(i)}
\\+\z^{(i)}(K^{-1}(t))^T\cdot\nabla\te=\displaystyle\sum_{i=1}^2\beta_{ii}\nabla\times \z^{(i)}+\sum_{i,j=1}^2(-1)^{j+1}\beta_{ij}\frac{\partial \z^{(i)}_i}{\partial y_{3-j}}
+\displaystyle u^{(i)}\chi_{\mathcal{O}_i\times[0,T]}\;\; \mbox{ in } Q
\\\nabla\cdot(M^{-1}(\z^{(i)})^T)=0 \hspace*{8,7cm} \;\; \mbox{ in }Q
\\\z^{(i)}=0,\;\;w^{(i)}=0 \hspace*{9,1cm} \;\; \mbox{ on } \Sigma
\\\z^{(i)}(0)= 0,\;\;\;\;w^{(i)}(0)= 0 \hspace*{8,08cm} \;\; \mbox{ in }\Omega\label{6}%
\end{array}\right.\end{equation}

We can write the solution $(\z,w,p)$ of $( \ref{eq10})$ as
\begin{equation}
\begin{array}{l}
(\z,w,p)  =(L_{1,1}\v^{(1)}+L_{1,2}\v^{(2)}+  \widetilde{L}_{1,1}u^{(1)}+\widetilde{ L}_{1,2}u^{(2)} + \q,\\
\hspace{1,8cm} {L}_{2,1}\v^{(1)} +{L}_{2,2}\v^{(2)}+ \widetilde{L}_{2,1}u^{(1)}+ \widetilde {L}_{2,2}u^{(2)} + r,\\
\hspace{1,8cm} p^{(1)} + p^{(2)} + s)\label{7}%
 \end{array}
\end{equation}
where $(\q,r,s)$ solves
\begin{equation}
\displaystyle\left|\begin{array}{l}
\displaystyle \q_t-yK'(t)K^{-1}(t)\nabla \q-\sum^{2}_{j,l,r=1}\beta_{lj}(t)
\beta_{rj}(t)\frac{\partial^2
\q}{{\partial y_l}{\partial y_r}}(y,t)+(\h(K^{-1}(t))^T)\cdot\nabla)\q
\\\displaystyle+(\q(K^{-1}(t))^T)\cdot\nabla)\h+\nabla s(y,t)K^{-1}(t)=\overline{K^{-1}(t)}\nabla\times r+\f\nchi_{\mathcal{O}\times[0,T]} \hspace{1,1 cm}\;\; \mbox { in } Q
\\\displaystyle r_t-yK'(t)K^{-1}(t)\nabla r-\sum^{2}_{j,l,r=1}\beta_{lj}(t)
\beta_{rj}(t)\frac{\partial^2
r}{{\partial y_l}{\partial y_r}}(y,t)+\h(K^{-1}(t))^T\cdot\nabla r
\\+\q(K^{-1}(t))^T\cdot\nabla\te=\displaystyle\sum_{i=1}^2\beta_{ii}\nabla\times \q+\sum_{i,j=1}^2(-1)^{j+1}\beta_{ij}\frac{\partial \q_i}{\partial y_{3-j}}
+\displaystyle g\chi_{\mathcal{O}\times[0,T]} \hspace{0,68cm} \;\; \mbox{ in } Q
\\\nabla\cdot(M^{-1}\q^T)=0 \hspace{8,9cm} \;\; \mbox{ in }Q
\\\q=0,\;\;r=0 \hspace{9,5cm} \;\; \mbox{ on } \Sigma
\\\q(0)=z^{0},\;\;\;\;r(0)=w^{0} \hspace{8cm}\;\; \mbox{ in }\Omega\label{8}%
\end{array}\right.\end{equation}

We can rewrite the functionals defined in $(\ref{3})$ and $(\ref{4})$ as
\begin{align}
 \nonumber J_{i}(\f, g;\v,u) &= \displaystyle\frac{\alpha_{i}}{2}\int_{0}%
^{T}\int_{\mathcal{O}_{i,d}}|\det K(t)|\left\vert L_{1,1}\v^{(1)
}+L_{1,2}\v^{(2)} + \widetilde{L}_{1,1}u^{(1)}+\widetilde{ L}_{1,2}u^{(2)} -(\z_{i,d}-\q)  \right\vert
^{2}dydt\\[15pt] &+\displaystyle\frac{\mu_{i}}{2}\int_{0}^{T}\int_{\mathcal{O}
_{i}}|\det K(t)|\vert \v^{(i)}\vert ^{2}dydt,\quad
i=1,2 \label{9}%
\end{align}
and%
\begin{align}
 \nonumber \widetilde{J}_{i}(\f, g;\v,u)&=\displaystyle\frac{\widetilde{\alpha}_{i}}%
{2}\int_{0}^{T}\int_{\mathcal{O}_{i,d}}|\det K(t)|\left\vert
{L}_{2,1}\v^{(1)} +{L}_{2,2}\v^{(2)} +  \widetilde{L}_{2,1}u^{(1)}+\widetilde{L}_{2,2}u^{(2)}-(w_{i,d}-r)
\right\vert ^{2}dydt\\[15pt]&+\displaystyle\frac{\widetilde{\mu}_{i}}{2}\int_{0}^{T}\int
_{\mathcal{O}_{i}}|\det K(t)|\vert u^{(i)}\vert ^{2}dydt, \quad
i=1,2.\label{10}%
\end{align}
Thus, by (\ref{en10}) and (\ref{en11}), the vector $((\v^{(1) },\v^{(2)})
,(u^{(1)},u^{(2)})) \in \hnc\times\mathcal{H}$ is a Nash equilibrium if
\begin{equation}
\alpha_{i}(L_{1,1}\v^{(1)} +L_{1,2}\v^{(2)} + \widetilde{L}_{1,1}u^{(1)}+\widetilde{ L}_{1,2}u^{(2)}-( \z
_{i,d}-\q) ,L_{1,i}\hat{{\v}}^{(i)}) _{\mathbf{L}^{2}(
\mathcal{O}_{i,d}\times(0,T))}+\mu_{i}( \v^{(i) },\hat{{\v}}^{(
i)}) _{\hnc_{i}}=0\label{11}
\end{equation}
and%
\begin{equation}
\widetilde{\alpha}_{i}({L}_{2,1}\v^{(
1)}+{L}_{2,2}\v^{(2)} +  \widetilde{L}_{2,1}u^{(1)}+\widetilde{L}_{2,2}u^{(2)}- ( w_{i,d}-r)
,\widetilde{L}_{2,i}\hat{{u}}^{(i)})_{L^{2}(\mathcal{O}%
_{i,d}\times(0,T))}+\widetilde{\mu}_{i}( u^{(i)},\hat{{{u}}}^{(i)})
_{\mathcal{H}_{i}}=0, \label{12}%
\end{equation}
 for $ i = 1,2 $ and for all $((\hat{{{\v}}}^{(1) },\hat{{{\v}}}^{(2)}) ,(\hat{{{u}}}^{(1)},\hat{{{u}}}^{(2)})) \in
\hnc\times\mathcal{H}$. Consequently
\begin{equation}
\alpha_{i}(L_{1,i}^{\ast}[
L_{1,1}\v^{(1)}+L_{1,2}\v^{(2)} +  \widetilde{L}_{1,1}u^{(1)}+\widetilde{L}_{1,2}u^{(2)}  -(
\z_{i,d}-\q)] ,\hat{{\v}}^{( i)})_{\mathbf{L}^{2}(
\mathcal{O}_{i,d}\times(0,T))}+\mu_{i}( \v^{(i)},\hat{{\v}}^{(
i)})_{\hnc_{i}}=0 \label{13}%
\end{equation}
and%
\begin{equation}
\widetilde{\alpha}_{i}(\widetilde{L}_{2,i}^{\ast}[{L}_{2,1}\v^{(
1)}+{L}_{2,2}\v^{(2)} +    \widetilde
{L}_{2,1}u^{(1) }+\widetilde{L}_{2,2}u^{(2)}-(
w_{i,d}-r)] ,\hat{{u}}^{(i)}) _{L^{2}(\mathcal{O}_{i,d}\times(0,T))
}+\widetilde{\mu}_{i}(u^{(i)
},\hat{{u}}^{(i)})_{\mathcal{H}_{i}}=0, \label{14}%
\end{equation}
for $ i = 1,2 $ and for all $((\hat{{\v}}^{(1) },\hat{{\v}}^{(2)})
,(\hat{{u}}^{(1)},\hat{{u}}^{(2)})) \in \hnc\times\mathcal{H}$, where
$L_{1,i}^{\ast}\in\mathcal{L}\left(\mathbf{L}^{2}\left( Q\right) ; \hnc_{i} \right)$ and
$\widetilde{L}_{2,i}^{\ast}\in\mathcal{L}\left(L^{2}\left( Q\right) ; \mathcal{H}_{i}
\right)$ are, respectively, the adjoint operators of $L_{1,i}$ and
$\widetilde{L}_{2,i}$ defined in  $(\ref{en44})$ and $(\ref{en55})$.

From $(\ref{13})$ and $(\ref{14})$ we  deduce, for  $i = 1,2$, that
\[
\alpha_{i}L_{1,i}^{\ast}[  L_{1,1}\v^{(
1)}+L_{1,2}\v^{(2)} + \widetilde{L}_{1,1}u^{(1)}+\widetilde{L}_{1,2}u^{(2)}]  \chi_{\mathcal{O}_{i,d}}+\mu_{i}%
\v^{(i) }=\alpha_{i}L_{1,i}^{\ast}(\z_{i,d}-\q)
\chi_{\mathcal{O}_{i,d}}\quad\text{in}\quad \hnc_{i},
\]
and%
\[
\widetilde{\alpha}_{i}\widetilde{L}_{2,i}^{\ast}[{L}_{2,1}\v^{(
1)}+{L}_{2,2}\v^{(2)}+ \widetilde{L}_{2,1}u^{(1)}+\widetilde{L}_{2,2}u^{(2)}]\chi_{\mathcal{O}_{i,d}}%
+\widetilde{\mu}_{i}u^{(i)}=\widetilde{\alpha}_{i}\widetilde{L}%
_{2,i}^{\ast}(w_{i,d}-r)\chi_{\mathcal{O}_{i,d}}\quad\text{in}\quad \mathcal{H}_{i}.
\]

If we consider the operator $\mathbb{L}\in\mathcal{L}\left(  \hnc \times \mathcal{H};\hnc \times \mathcal{H}\right)  $ defined as%
\[%
\begin{array}
[c]{ll}%
\mathbb{L}((\v^{(1)},\v^{(2)}),(u^{(1},u^{(2)}))= & (\alpha_{1}L_{1,1}^{\ast
}[L_{1,1}\v^{(1)}+L_{1,2}\v^{(2)} +  \widetilde{L}_{1,1}u^{(1)}+\widetilde{L}_{1,2}u^{(2)}]\chi_{\mathcal{O}_{1,d}}+\mu_{1}\v^{(1)},\smallskip\\
& \alpha_{2}L_{1,2}^{\ast}[L_{1,1}\v^{(1)}+L_{1,2}\v^{(2)} +  \widetilde{L}_{1,1}u^{(1)}+\widetilde{L}_{1,2}u^{(2)}]\chi_{\mathcal{O}_{2,d}}%
+\mu_{2}\v^{(2)},\smallskip\\
& \widetilde{\alpha}_{1}\widetilde{L}_{2,1}^{\ast}[{L}_{2,1}\v^{(
1)}+{L}_{2,2}\v^{(2)} + \widetilde{L}_{2,1}u^{(1)}+\widetilde{L}_{2,2}u^{(2)}]\chi_{\mathcal{O}_{1,d}}%
+\widetilde{\mu}_{1}u^{(1)},\smallskip\\
& \widetilde{\alpha}_{2}\widetilde{L}_{2,2}^{\ast}[{L}_{2,1}\v^{(
1)}+{L}_{2,2}\v^{(2)} +  \widetilde{L}_{2,1}u^{(1)}+\widetilde{L}_{2,2}u^{(2)}]\chi_{\mathcal{O}_{2,d}}%
+\widetilde{\mu}_{2}u^{(2)}),
\end{array}
\]

then $\Xi=((\v^{(1)},\v^{(2)}),(u^{(1)},u^{(2)}))\in\hnc \times \mathcal{H}$ is a Nash equilibrium if%
\begin{equation}
\mathbb{L}\left(  \Xi\right)  =\Psi,\label{en1}%
\end{equation}
with%
\[%
\begin{array}
[c]{ll}%
\Psi=(\alpha_{1}L_{1,1}^{\ast}[(\z_{1,d}-\q)\chi_{\mathcal{O}_{1,d}%
},\alpha_{2}L_{1,2}^{\ast}[(\z_{2,d}-\q)\chi_{\mathcal{O}_{2,d}},\widetilde{\alpha
}_{i}\widetilde{L}_{2,1}^{\ast}[(w_{1,d}-r)\chi_{\mathcal{O}_{1,d}},\widetilde
{\alpha}_{2}\widetilde{L}_{2,2}^{\ast}[(w_{2,d}-r)\chi_{\mathcal{O}_{2,d}}).
\end{array}
\]

The following result  holds :

\begin{proposition} \label{pro2.1} Assume that
\begin{equation}
\begin{array}{l}
\alpha_{1}\Vert L_{1,2}\Vert^{2} + (\widetilde{\alpha}_{1} + \widetilde{\alpha}_{2})\Vert L_{2,2}\Vert^{2} <\frac {4}{3}\mu_{2},\quad\alpha_{2}\Vert L_{1,1}\Vert^{2} + (\widetilde{\alpha}_{1} + \widetilde{\alpha}_{2})\Vert L_{2,1}\Vert^{2}
<\frac {4}{3}\mu_{1}, \\[10pt]
 \widetilde{\alpha}_{1}\Vert \widetilde{L}_{2,2}\Vert^{2} + ({\alpha}_{1} + {\alpha}_{2}) \Vert\widetilde {L}_{1,2}\Vert^{2} <\frac {4}{3}\widetilde{\mu}_{2},\quad \widetilde{\alpha}_{2}\Vert \widetilde{L}_{2,1}\Vert^{2} + ({\alpha}_{1} +
{\alpha}_{2})\Vert \widetilde{L}_{1,1}\Vert^{2} <\frac {4}{3}\widetilde{\mu}_{1},
\label{en2}%
\end{array}
\end{equation}
where $\|\cdot\|$ denotes the norm of the corresponding linear operator. Then
$\mathbb{L}$ is an invertible operator. In particular, for each $(\f,g) \in
\textbf{L}^{2}\left(  \mathcal{O}\times\left(  0,T\right)  \right) \times L^{2}\left(
\mathcal{O}\times\left(  0,T\right)  \right)  $ there exists an unique Nash
equilibrium $((\v^{(1)}(\f,g),\v^{(2)}(\f,g)),(u^{(1)}(\f,g),u^{(2)}(\f,g)))$ in the sense of
(\ref{en6}) and (\ref{en7}).
\end{proposition}

\begin{proof} We observe that
\[%
\begin{array}
[c]{l}%
(\mathbb{L}((\v^{(1)},\v^{(2)}),(u^{(1},u^{(2)}))),(\v^{(1)},\v^{(2)}%
),(u^{(1)},u^{(2)})))_{\hnc\times\mathcal{H}}\medskip\\
=\mu_{1}\Vert \v^{(1)}\Vert_{\hnc_{1}}^{2}+\alpha_{1}(\sum
\limits_{j=1}^{2}L_{1,j}\v^{(j)},L_{1,1}\v^{(1)})_{\textbf{L}^{2}(\mathcal{O}_{1,d}%
\times(0,T))} + \alpha_{1}(\sum
\limits_{j=1}^{2}\widetilde{L}_{1,j}u^{(j)},{L}_{1,1}\v^{(1)})_{\textbf{L}^{2}(\mathcal{O}_{1,d}%
\times(0,T))}\medskip\\ +\mu_{2}\Vert \v^{(2)}\Vert_{\hnc_{2}}^{2}
+\alpha_{2}(\sum\limits_{j=1}^{2}L_{1,j}\v^{(j)},L_{1,2}\v^{(2)})_{\textbf{L}^{2}%
(\mathcal{O}_{2,d}\times(0,T))} + \alpha_{2}(\sum
\limits_{j=1}^{2}\widetilde{L}_{1,j}u^{(j)},{L}_{1,2}\v^{(2)})_{\textbf{L}^{2}(\mathcal{O}_{2,d}%
\times(0,T))}\medskip\\ +\widetilde{\mu}_{1}\Vert u^{(1)}%
\Vert_{\mathcal{H}_{1}}^{2} + \widetilde{\alpha}_{1}(\sum
\limits_{j=1}^{2}L_{2,j}\v^{(j)},\widetilde{L}_{2,1}u^{(1)})_{L^{2}(\mathcal{O}_{1,d}%
\times(0,T))} +\widetilde{\alpha}_{1}(\sum
\limits_{j=1}^{2}\widetilde{L}_{2,j}u^{(j)},\widetilde{L}_{2,1}u^{(1)})_{L^{2}(\mathcal{O}_{1,d}%
\times(0,T))} \medskip\\+ \widetilde{\mu}_{2}\Vert u^{(2)}%
\Vert_{\mathcal{H}_{2}}^{2} + \widetilde{\alpha}_{2}(\sum
\limits_{j=1}^{2}L_{2,j}\v^{(j)},\widetilde{L}_{2,2}u^{(2)})_{L^{2}(\mathcal{O}_{2,d}%
\times(0,T))} +\widetilde{\alpha}_{2}(\sum
\limits_{j=1}^{2}\widetilde{L}_{2,j}u^{(j)},\widetilde{L}_{2,2}u^{(2)})_{L^{2}(\mathcal{O}_{2,d},%
\times(0,T))}, \medskip\\
\end{array}
\]
that is,
\begin{equation}\label{1000}
\begin{array}
[c]{l}%
(\mathbb{L}((\v^{(1)},\v^{(2)}),(u^{(1},u^{(2)}))),(\v^{(1)},\v^{(2)}%
),(u^{(1)},u^{(2)})))_{\hnc\times\mathcal{H}}\medskip\\
=\mu_{1}\Vert
\v^{(1)}\Vert_{\mathcal{H}_{1}}^{2}+\mu_{2}\Vert \v^{(2)}\Vert_{\mathcal{H}_{2}}^{2}+\alpha_{1}\Vert L_{1,1}\v^{(1)}\Vert_{\textbf{L}^{2}(\mathcal{O}_{1,d}\times
(0,T))}^{2} + \alpha_{1}(L_{1,2}\v^{(2)},L_{1,1}\v^{(1)})_{\textbf{L}^{2}(\mathcal{O}_{1,d}\times
(0,T))} \medskip\\
+ \alpha_{1}(\sum
\limits_{j=1}^{2}\widetilde{L}_{1,j}u^{(j)},{L}_{1,1}\v^{(1)})_{\textbf{L}^{2}(\mathcal{O}_{1,d}%
\times(0,T))}+ \alpha_{2}\Vert L_{1,2}\v^{(2)}\Vert_{\textbf{L}^{2}(\mathcal{O}_{2,d}%
\times(0,T))}^{2}\medskip\\
+\alpha_{2}(L_{1,1}\v^{(1)},L_{1,2}\v^{(2)})_{\textbf{L}^{2}(\mathcal{O}_{2,d}%
\times(0,T))}+    \alpha_{2}(\sum
\limits_{j=1}^{2}\widetilde{L}_{1,j}u^{(j)},{L}_{1,2}\v^{(2)})_{\textbf{L}^{2}(\mathcal{O}_{2,d}%
\times(0,T))}\medskip\\
+\widetilde{\mu}_{1}\Vert u^{(1)}\Vert_{\mathcal{H}_{1}}^{2}+  \widetilde{\mu}_{2}\Vert u^{(2)}\Vert_{\mathcal{H}_{2}}^{2} +\widetilde{\alpha}_{1}%
\Vert\widetilde{L}_{2,1}u^{(1)}\Vert_{L^{2}(\mathcal{O}_{1,d}\times(0,T))}%
^{2}+\widetilde{\alpha}_{1}(\widetilde{L}_{2,2}u^{(2)},\widetilde{L}_{2,1}%
u^{(1)})_{L^{2}(\mathcal{O}_{1,d}\times(0,T))} \medskip\\
+ \widetilde{\alpha}_{1}(\sum
\limits_{j=1}^{2}{L}_{2,j}\v^{(j)},\widetilde{L}_{2,1}u^{(1)})_{L^{2}(\mathcal{O}_{1,d}%
\times(0,T))} + \widetilde{\alpha}_{2}\Vert \widetilde{L}_{2,2}u^{(2)}\Vert_{L^{2}(\mathcal{O}_{2,d}%
\times(0,T))}^{2}\medskip\\
+\widetilde{\alpha}_{2}(\widetilde{L}%
_{2,1}u^{(1)},\widetilde{L}_{2,2}u^{(2)})_{L^{2}(\mathcal{O}_{2,d}\times(0,T))} + \widetilde{\alpha}_{2}(\sum
\limits_{j=1}^{2}\widetilde{L}_{2,j}\v^{(j)},\widetilde{L}_{2,2}u^{(2)})_{L^{2}(\mathcal{O}_{2,d},%
\times(0,T))}.
\end{array}
\end{equation}

 The Young's inequality implies that
\begin{equation}\label{mat1}
\begin{array}{rcl}
\disp \alpha_{1}(L_{1,2}\v^{(2)},L_{1,1}\v^{(1)})_{\textbf{L}^{2}(\mathcal{O}_{1,d}\times
(0,T)) }&\geq&    - \disp \frac{\alpha_{1}}{3}\Vert L_{1,1}\v^{(1)}\Vert_{\textbf{L}^{2}(\mathcal{O}_{1,d}%
\times(0,T))}^{2}\\ [10pt]
&-& \disp  \frac{3\alpha_{1}}{4}\Vert L_{1,2}\v^{(2)}\Vert_{\textbf{L}^{2}%
(\mathcal{O}_{1,d}\times(0,T))}^{2},
\end{array}
\end{equation}

\begin{equation}\label{mat2}
\begin{array}{rcl}
\disp \alpha_{1}(\sum
\limits_{j=1}^{2}\widetilde{L}_{1,j}u^{(j)},{L}_{1,1}\v^{(1)})_{\textbf{L}^{2}(\mathcal{O}_{1,d}%
\times(0,T))} &\geq&   - \disp \frac{2\alpha_{1}}{3}\Vert L_{1,1}\v^{(1)}\Vert_{\textbf{L}^{2}(\mathcal{O}_{1,d}%
\times(0,T))}^{2}\\ [10pt]
 &-&\disp \frac{3\alpha_{1}}{4}\Vert \widetilde{L}_{1,1}u^{(1)}\Vert_{\textbf{L}^{2}(\mathcal{O}_{1,d}%
\times(0,T))}^{2}\\[10pt]
 &-& \disp \frac{3\alpha_{1}}{4}\Vert \widetilde{L}_{1,2}u^{(2)}\Vert_{\textbf{L}^{2}(\mathcal{O}_{1,d}%
\times(0,T))}^{2},
\end{array}
\end{equation}

\begin{equation}\label{mat3}
\begin{array}{rcl}
\disp \alpha_{2}(L_{1,1}\v^{(1)},L_{1,2}\v^{(2)})_{L^{2}(\mathcal{O}_{1,d}\times
(0,T))} &\geq&  - \disp \frac{\alpha_{2}}{3}\Vert L_{1,2}\v^{(2)}\Vert_{\textbf{L}^{2}%
(\mathcal{O}_{2,d}\times(0,T))}^{2}\\  [10pt]
&-& \disp \frac{3\alpha_{2}}{4}\Vert L_{1,1}\v^{(1)}\Vert_{\textbf{L}^{2}%
(\mathcal{O}_{2,d}\times(0,T))}^{2},\\
\end{array}
\end{equation}

\begin{equation}\label{mat4}
\begin{array}{rcl}
\alpha_{2}(\sum
\limits_{j=1}^{2}\widetilde{L}_{1,j}u^{(j)},{L}_{1,2}\v^{(2)})_{\textbf{L}^{2}(\mathcal{O}_{2,d}%
\times(0,T))} &\geq&  -\disp \frac{2\alpha_{2}}{3}\Vert L_{1,2}\v^{(2)}\Vert_{\textbf{L}^{2}(\mathcal{O}_{2,d}%
\times(0,T))}^{2}\\ [10pt]
&-&\disp \frac{3\alpha_{2}}{4}\Vert \widetilde{L}_{1,1}u^{(1)}\Vert_{\textbf{L}^{2}(\mathcal{O}_{2,d}%
\times(0,T))}^{2}\\[10pt]
&-& \disp \frac{3\alpha_{2}}{4}\Vert \widetilde{L}_{1,2}u^{(2)}\Vert_{\textbf{L}^{2}(\mathcal{O}_{2,d}%
\times(0,T))}^{2},
\end{array}
\end{equation}

\begin{equation}\label{mat5}
\begin{array}{rcl}
\widetilde{\alpha}_{1}(\widetilde{L}_{2,2}u^{(2)},\widetilde{L}_{2,1}%
u^{(1)})_{L^{2}(\mathcal{O}_{1,d}\times(0,T))} &\geq&  - \disp \frac {\widetilde{\alpha}_{1}}{3}
\Vert\widetilde{L}_{2,1}u^{(1)}\Vert_{L^{2}(\mathcal{O}_{1,d}\times(0,T))}%
^{2}\\[10pt]
&-& \disp \frac{3\widetilde{\alpha}_{1}}{4}\Vert\widetilde{L}_{2,2}u^{(2)}\Vert
_{L^{2}(\mathcal{O}_{1,d}\times(0,T))}^{2},%
\end{array}
\end{equation}

\begin{equation}\label{mat6}
\begin{array}{rcl}
\widetilde{\alpha}_{1}(\sum
\limits_{j=1}^{2}L_{2,j}\v^{(j)},\widetilde{L}_{2,1}u^{(1)})_{L^{2}(\mathcal{O}_{1,d}%
\times(0,T))} &\geq&  - \disp \frac{2\widetilde{\alpha}_{1}}{3}\Vert \widetilde{L}_{2,1}u^{(1)}\Vert_{L^{2}(\mathcal{O}_{1,d}%
\times(0,T))}^{2}\\ [10pt]
&-&\disp \frac{3\widetilde{\alpha}_{1}}{4}\Vert {L}_{2,1}\v^{(1)}\Vert_{L^{2}(\mathcal{O}_{1,d}%
\times(0,T))}^{2}\\[10pt]
&-& \disp \frac{3\widetilde{\alpha}_{1}}{4}\Vert \widetilde{L}_{2,2}\v^{(2)}\Vert_{L^{2}(\mathcal{O}_{1,d}%
\times(0,T))}^{2},
\end{array}
\end{equation}

\begin{equation}\label{mat7}
\begin{array}{rcl}
\widetilde{\alpha}_{2}(\widetilde{L}_{2,1}u^{(1)},\widetilde{L}_{2,2}%
u^{(2)})_{L^{2}(\mathcal{O}_{2,d}\times(0,T))} &\geq& - \disp \frac {\widetilde{\alpha}_{2}}{3}
\Vert\widetilde{L}_{2,2}u^{(2)}\Vert_{L^{2}(\mathcal{O}_{2,d}\times(0,T))}%
^{2}\\[10pt]
&-& \disp \frac{3\widetilde{\alpha}_{2}}{4}\Vert\widetilde{L}_{2,1}u^{(1)}\Vert
_{L^{2}(\mathcal{O}_{2,d}\times(0,T))}^{2},%
\end{array}
\end{equation}

and
\begin{equation}\label{mat8}
\begin{array}{rcl}
\disp \widetilde{\alpha}_{2}(\sum
\limits_{j=1}^{2}L_{2,j}\v^{(j)},\widetilde{L}_{2,2}u^{(2)})_{L^{2}(\mathcal{O}_{2,d}%
\times(0,T))} &\geq&   - \disp \frac{2\widetilde{\alpha}_{2}}{3}\Vert \widetilde{L}_{2,2}u^{(2)}\Vert_{L^{2}(\mathcal{O}_{2,d}%
\times(0,T))}^{2}\\ [10pt]
&-&\disp \frac{3\widetilde{\alpha}_{2}}{4}\Vert {L}_{2,1}\v^{(1)}\Vert_{L^{2}(\mathcal{O}_{2,d}%
\times(0,T))}^{2}\\[10pt]
&-& \disp \frac{3\widetilde{\alpha}_{2}}{4}\Vert \widetilde{L}_{2,2}\v^{(2)}\Vert_{L^{2}(\mathcal{O}_{2,d}%
\times(0,T))}^{2}.
\end{array}
\end{equation}

Substituting  (\ref {mat1})--(\ref{mat8}) in (\ref{1000}), we get
\begin{equation}
\begin{array}
[c]{l}%
(\mathbb{L}((\v^{(1)},\v^{(2)}),(u^{(1)},u^{(2)})),((\v^{(1)},\v^{(2)}%
),(u^{(1)},u^{(2)})))_{\mathcal{H}\times\mathcal{H}}\medskip\\
\geq\gamma(\Vert
\v^{(1)}\Vert_{\mathcal{H}_{1}}^{2}+\Vert v^{(2)}\Vert_{\mathcal{H}_{2}}%
^{2}+\Vert u^{(1)}\Vert_{\mathcal{H}_{1}}^{2}+\Vert u^{(2)}\Vert
_{\mathcal{H}_{2}}^{2})\medskip\\
=\gamma\Vert((\v^{(1)},\v^{(2)}),(u^{(1)},u^{(2)}))\Vert_{\mathcal{H}%
\times\mathcal{H}}^{2},\label{en3}
\end{array}
\end{equation}
where $\gamma > 0$  (see (\ref {en2}))  is given by
\begin{equation*}
\begin{array}{rcl}
 \gamma = \min \{ \mu_1 &-&  \frac{3}{4} [ \alpha_{2}\|L_{1,1} \|^2  + (\widetilde{\alpha}_{1} + \widetilde{\alpha}_{2})  \|L_{2,1}\|^{2}],\; \mu_2 -
\frac{3}{4} [ \alpha_{1}\|L_{1,2} \|^2  + (\widetilde{\alpha}_{1} + \widetilde{\alpha}_{2})  \|L_{2,2}\|^{2}],  \\[10pt]
\widetilde{\mu}_1 &-& \frac{3}{4} [\widetilde{\alpha}_{2}\|\widetilde{L}_{2,1} \|^2  + ({\alpha}_{1} + {\alpha}_{2})  \|\widetilde{L}_{1,1}\|^{2}], \;\widetilde{\mu}_2 -  \frac{3}{4} [ \widetilde{\alpha}_{1}\|\widetilde{L}_{2,2} \|^2
+ ({\alpha}_{1} + {\alpha}_{2})  \|\widetilde{L}_{2,1}\|^{2}]\}.
\end{array}
\end{equation*}

Now, we define the functional $a: (\hnc\times \mathcal{H})\times (\hnc\times
\mathcal{H}) \rightarrow\mathbb{R}$ by
\[%
\begin{array}
[c]{l}%
a((\v^{(1)},\v^{(2)}),(u^{(1)},u^{(2)})),(\widetilde{\v}^{(1)},\widetilde{{\v}}%
^{(2)}),(\widetilde{u}^{(1)},\widetilde{u}^{(2)}))\medskip\\
=(\mathbb{L}((\v^{(1)},\v^{(2)}),(u^{(1)},u^{(2)})),((\widetilde{\v}^{(1)}%
,\widetilde{{\v}}^{(2)}),(\widetilde{u}^{(1)},\widetilde{u}^{(2)}))_{\hnc\times\mathcal{H}}.
\end{array}
\]

Obviously, from the definition of the operator $\mathbb{L}$ and the inequality
(\ref{en3}), $a(\cdot,\cdot)$ is a coercive continuous bilinear form. Consequently,
the Lax-Milgram Theorem implies that for all $\varphi \in (\hnc\times
\mathcal{H})^{'}$
\begin{eqnarray}
&& a(\widetilde{x}, \widetilde{y})= \langle\varphi, \widetilde{y}\rangle_{(\hnc\times \mathcal{H})^{'}\times (\hnc\times \mathcal{H})} , \,\, \forall \, \widetilde{y} \in \hnc\times \mathcal{H},
\end{eqnarray}
with $\widetilde{x} = (( \v^{(1)},\v^{(2)}), (u^{(1)}, u^{(2)}))$ and $\widetilde{y}= (( \widetilde{\v}^{(1)},\widetilde{\v}^{(2)}), (\widetilde{u}^{(1)}, \widetilde{u}^{(2)}))$.

On the other hand, since
$$
\langle \varphi, \widetilde{y}\rangle_{(\hnc\times \mathcal{H})^{'}\times (\hnc\times \mathcal{H})} = (\mathbb{L}\widetilde{x}, \widetilde{y})_{\hnc\times \mathcal{H}},
$$
the proof  is completed.
\end{proof}

\subsection{ Characterization of Nash equilibrium}

We observe that the characterization (\ref{en10}), (\ref{en11}) of the Nash
equilibrium can be written as follows: $(\v,u)=((\v^{(1)},\v^{(2)}),(u^{(1)},u^{(2)}))$ is
a Nash equilibrium if
\begin{equation}
\left\{
\begin{array}
[c]{l}%
\alpha_{i}(\z-\z_{i,\alpha},\be^{(i)})_{\textbf{L}^{2}(\mathcal{O}_{i,d}\times
(0,T))}+\mu_{i}(\v^{(i)},\hat{{\v}}^{(i)})_{\hnc_{i}}=0,\\
\widetilde{\alpha}_{i}(w-w_{i,\alpha},\gamma^{(i)})_{L^{2}(\mathcal{O}_{i,d
}\times(0,T))}+\widetilde{\mu}_{i}(u^{(i)},\hat{{u}}^{(i)})_{\mathcal{H}_{i}}=0,%
\end{array}
\right.  \label{e2.12}%
\end{equation}
for $ i = {1,2}$ and for all $(\hat{{\v}}^{(i)},\hat{{u}}^{(i)})\in \hnc_{i} \times
\mathcal{H}_{i}$, where $( \be^{(i)},\gamma^{(i)}, \bar{p}^{(i)})$  is the  solution of
the problem
\begin{equation}
\displaystyle\left|\begin{array}{l}
\displaystyle \be^{(i)}_t-yK'(t)K^{-1}(t)\nabla \be^{(i)}-\sum^{2}_{j,l,r=1}\beta_{lj}(t)
\beta_{rj}(t)\frac{\partial^2 \be^{(i)}}{{\partial y_l}{\partial
y_r}}(y,t)+(\overline{\h}(K^{-1}(t))^T)\cdot\nabla)\be^{(i)}
\\\displaystyle+(\be^{(i)}(K^{-1}(t))^T)\cdot\nabla)\overline{\h}+\nabla \overline{p}^{(i)}(y,t)K^{-1}(t)=\overline{K^{-1}(t)}\nabla\times \gamma^{(i)}+\hat{\v}^{(i)}\nchi_{\mathcal{O}_i\times[0,T]}
\hspace*{.5
cm}\mbox { in } Q,
\\\displaystyle \gamma^{(i)}_t-yK'(t)K^{-1}(t)\nabla \gamma^{(i)}-\sum^{2}_{j,l,r=1}\beta_{lj}(t)
\beta_{rj}(t)\frac{\partial^2
\gamma^{(i)}}{{\partial y_l}{\partial y_r}}(y,t)+\overline{\h}(K^{-1}(t))^T\cdot\nabla \gamma^{(i)}
\\+\be^{(i)}(K^{-1}(t))^T\cdot\nabla\overline{\te}\!=\!\displaystyle\sum_{i=1}^2\beta_{ii}(t)\nabla\times \be^{(i)}\!\!+\!\!\sum_{i,j=1}^2(-1)^{j+1}\beta_{ij}(t)\frac{\partial \be^{(i)}_i}{\partial y_{3-j}}
\!+\!\displaystyle \hat{u}^{(i)}\chi_{\mathcal{O}_i\times[0,T]}\mbox{ in } Q,
\\\nabla\cdot(M^{-1}(\be^{(i)})^T)=0\hspace*{9.1cm}\mbox{ in }Q,
\\\be^{(i)}=0,\;\;\gamma^{(i)}=0\hspace*{9,6cm}\mbox{ on } \Sigma,
\\\be^{(i)}(0)=0,\;\;\;\;\gamma^{(i)}(0)=0\hspace*{8,6cm} \mbox{ in }\Omega.\label{e2.13}%
\end{array}\right.\end{equation}
with $(\overline{\h},\overline{\theta})$ in the class (\ref{en14}).

To express (\ref {e2.12}) in a suitable form, we introduce the adjoint state
$(\q^{(i)}, r^{(i)}, {\pi}^{(i)})$ $(i=1,2)$ as a solution of the problem%
\begin{equation}
\displaystyle\left|\begin{array}{l}
\displaystyle -\q^{(i)}_t+yK'(t)K^{-1}(t)\nabla \q^{(i)}-\sum^{2}_{j,l,r=1}\beta_{lj}(t)
\beta_{rj}(t)\frac{\partial^2
\q^{(i)}}{{\partial y_l}{\partial y_r}}(y,t)
\displaystyle-(K^{-1}(t))^T)D\q^{(i)}\overline{\h}\\\displaystyle+\nabla {\pi}^{(i)}(y,t)K^{-1}(t)=\overline{K^{-1}(t)}\nabla\times r^{(i)}+\overline{\te}^{(i)}(K^{-1}(t))^T\cdot\nabla r+\alpha_{i}(\z-\z_{i,\alpha})\nchi_{\mathcal{O}_{i,d}}
\hspace*{0,2cm} \mbox { in } Q,
\\\displaystyle- r^{(i)}_t+yK'(t)K^{-1}(t)\nabla r^{(i)}-\sum^{2}_{j,l,r=1}\beta_{lj}(t)
\beta_{rj}(t)\frac{\partial^2
r^{(i)}}{{\partial y_l}{\partial y_r}}(y,t)-\overline{\h}(K^{-1}(t))^T\cdot\nabla r^{(i)}
\\=\displaystyle\sum_{i=1}^2\beta_{ii}(t)\nabla\times \q^{(i)}+\sum_{i,j=1}^2(-1)^{j+1}\beta_{ij}(t)\frac{\partial \q^{(i)}_i}{\partial y_{3-j}}
+\displaystyle \widetilde{\alpha}_{i}(w-w_{i,\alpha})\nchi_{\mathcal{O}_{i,d}}\hspace*{2cm}\mbox{ in } Q,
\\\nabla\cdot(M^{-1}(\q^{(i)})^T)=0\hspace*{9.1cm}\mbox{ in }Q,
\\\q^{(i)}=0,\;\;r^{(i)}=0\hspace*{9,7cm}\mbox{ on } \Sigma,
\\\q^{(i)}(T)=0,\;\;\;\;r^{(i)}(T)=0\hspace*{8,5cm} \mbox{ in }\Omega.\label{e2.14}%
\end{array}\right.\end{equation}

Here, $D \varphi$ stands for the symmetric gradient of $\varphi$:
$$
D \varphi = \nabla \varphi + \nabla \varphi^t
$$

Multiplying $(\ref{e2.14})_{1}$ by $\be^{(i)}$, $(\ref{e2.14})_{2}$ by $\gamma^{(i)}$,
integration by parts, leads to :
\begin{equation}
\displaystyle\left\{\begin{array}{l}
\!\!\!\!\!\left(\q^{(i)},\displaystyle \be^{(i)}_t-yK'(t)K^{-1}(t)\nabla \be^{(i)}-\sum^{2}_{j,l,r=1}\beta_{lj}(t)
\beta_{rj}(t)\frac{\partial^2
\be^{(i)}}{{\partial y_l}{\partial y_r}}(y,t)+(\overline{\h}(K^{-1}(t))^T)\cdot\nabla)\be^{(i)}\right.
\\\left.\displaystyle+(\be^{(i)}(K^{-1}(t))^T)\cdot\nabla\overline{\h}-\left(\overline{K^{-1}(t)}\nabla\times \gamma^{(i)}\right)\right)_{\L^2(Q)}=\displaystyle\alpha_{i}\left(\z-\z_{i,d},\be^{(i)}\right)_{\L^2(\mathcal{O}_{i,d}\times(0,T))}
\\\!\!\!\!\!\left(r^{(i)},\displaystyle \gamma^{(i)}_t-yK'(t)K^{-1}(t)\nabla \gamma^{(i)}-\sum^{2}_{j,l,r=1}\beta_{lj}(t)
\beta_{rj}(t)\frac{\partial^2
\gamma^{(i)}}{{\partial y_l}{\partial y_r}}(y,t)+\overline{\h}(K^{-1}(t))^T\cdot\nabla \gamma^{(i)}\right.
\\+\left.\be^{(i)}(K^{-1}(t))^T\cdot\nabla\overline{\te}\!-\!\displaystyle\sum_{i=1}^2\beta_{ii}(t)\nabla\times \be^{(i)}\!\!-\!\!\sum_{i,j=1}^2(-1)^{j+1}\beta_{ij}(t)\frac{\partial \be^{(i)}_i}{\partial y_{3-j}}
\!\right)_{L^2(Q)}\\= \displaystyle \widetilde{\alpha}_{i}\left( w-w_{i,d},\gamma^{(i)}\right)_{L^2({\mathcal{O}_{i,d}}\times(0,T))}
\label{e2.15}%
\end{array}\right.\end{equation}

In the other hand, multiplying $(\ref{e2.13})_{1}$ by $\q^{(i)}$ and $(\ref{e2.13})_{2}$ by $r^{(i)}$, it follows that
\begin{equation}
\displaystyle\left\{\begin{array}{l}
\!\!\!\!\!\left(\q^{(i)},\displaystyle \be^{(i)}_t-yK'(t)K^{-1}(t)\nabla \be^{(i)}-\sum^{2}_{j,l,r=1}\beta_{lj}(t)
\beta_{rj}(t)\frac{\partial^2
\be^{(i)}}{{\partial y_l}{\partial y_r}}(y,t)+(\overline{\h}(K^{-1}(t))^T)\cdot\nabla)\be^{(i)}\right.
\\\left.\displaystyle+(\be^{(i)}(K^{-1}(t))^T)\cdot\nabla\overline{\h}-\left(\overline{K^{-1}(t)}\nabla\times \gamma^{(i)}\right)\right)_{\L^2(Q)}=\displaystyle\left(\q^{(i)},\hat{v}^{(i)}\right)_{\hnc_{i}}
\\\!\!\!\!\!\left(r^{(i)},\displaystyle \gamma^{(i)}_t-yK'(t)K^{-1}(t)\nabla \gamma^{(i)}-\sum^{2}_{j,l,r=1}\beta_{lj}(t)
\beta_{rj}(t)\frac{\partial^2
\delta^{(i)}}{{\partial y_l}{\partial y_r}}(y,t)+\overline{\h}(K^{-1}(t))^T\cdot\nabla \gamma^{(i)}\right.
\\+\left.\be^{(i)}(K^{-1}(t))^T\cdot\nabla\overline{\te}\!-\!\displaystyle\sum_{i=1}^2\beta_{ii}(t)\nabla\times \be^{(i)}\!\!-\!\!\sum_{i,j=1}^2(-1)^{j+1}\beta_{ij}(t)\frac{\partial \be^{(i)}_i}{\partial y_{3-j}}
\!\right)_{L^2(Q)}\\=\displaystyle\left(r^{(i)},\hat{u}^{(i)}\right)_{\mathcal{H}_{i}}
\label{e2.16}%
\end{array}\right.\end{equation}

Therefore, by comparing (\ref{e2.15}) and (\ref{e2.16}), we get
\begin{equation}
\left\{
\begin{array}
[c]{l}%
\alpha_i(\z - \z_{i,\alpha}, \be^{(i)})_{\textbf{L}^2(\mathcal{O}_{i,d} \times (0,T))} = (\q^{(i)},\hat{\v}^{(i)})_{\hnc_{i}},\\
\widetilde{\alpha}_i (w - w_{i,\alpha}, \gamma^{(i)})_{L^2(\mathcal{O}_{i,d} \times (0,T))}= (r^{(i)}, \hat{u}^{(i)})_{\mathcal{H}_{i}},
\end{array}
\right.  \label{e2.17}%
\end{equation}
for $ i = {1,2}$ and for all $(\hat{\v}^{(i)},\hat{u}^{(i)})\in \hnc_{i} \times
\mathcal{H}_{i}.$

Combining  (\ref{e2.12}) and (\ref{e2.17}), we find
\begin{equation}
\v^{(i)}=-\dfrac{1}{\mu_{i}}\q^{(i)}\nchi_{\mathcal{O}_{i}}\quad
\text{and}\quad u^{(i)}=-\dfrac{1}{\widetilde{\mu}_{i}}r^{(i)}\chi_{\mathcal{O}_{i}},\quad
i=1,2.\label{e2.19}%
\end{equation}
Summarizing, \ \ \ given $(\f, g)  \ \ \ \in  \ \ \textbf{L}^{2}(\mathcal{O} \times(0,T))\
\ \ \times \ \ \ \ L^{2}(\mathcal{O}\times(0,T))$, the pair $(\v,u)=((\v^{(1)},\v^{(2)}),
(u^{(1)},u^{(2)}))$ is a Nash equilibrium in the sense of (\ref{en6}), (\ref{en7}), if
and only if (\ref{e2.19}) holds, where
 $(\z,w,p,\q^{(1)},r^{(1)},{\pi}^{(1)},\q^{(2)},r^{(2)},{\pi}^{(2)}
)$ is the solution of the coupled system
\begin{equation}
\displaystyle\left|\begin{array}{l}
\displaystyle \z_t-yK'(t)K^{-1}(t)\nabla \z-\sum^{2}_{j,l,r=1}\beta_{lj}(t)
\beta_{rj}(t)\frac{\partial^2 \z}{{\partial y_l}{\partial
y_r}}(y,t)+(\h(K^{-1}(t))^T)\cdot\nabla)\z
\\\displaystyle+(\z(K^{-1}(t))^T)\cdot\nabla)\h+\nabla p(y,t)K^{-1}(t)=\overline{K^{-1}(t)}\nabla\times w+\f\nchi_{\mathcal{O}}\\-\displaystyle\frac{1}{\mu_1}\q^{(1)}\nchi_{\mathcal{O}_1}-\frac{1}{\mu_2}\q^{(2)}\nchi_{\mathcal{O}_2}
\;\; \hspace*{8,1cm}\mbox { in }\;\;  Q,
\\\displaystyle w_t-yK'(t)K^{-1}(t)\nabla w-\sum^{2}_{j,l,r=1}\beta_{lj}(t)
\beta_{rj}(t)\frac{\partial^2
w}{{\partial y_l}{\partial y_r}}(y,t)+\h(K^{-1}(t))^T\cdot\nabla w
\\+\z(K^{-1}(t))^T\cdot\nabla\te=\displaystyle\sum_{i=1}^2\beta_{ii}\nabla\times \z+\sum_{i,j=1}^2(-1)^{j+1}\beta_{ij}\frac{\partial \z_i}{\partial y_{3-j}}
\\+g\chi_{\mathcal{O}}\displaystyle -\dfrac{1}{\widetilde{\mu}_{1}}r^{(1)}\chi_{\mathcal{O}_{1}}-\dfrac{1}{\widetilde{\mu}_{2}}r^{(2)}\chi_{\mathcal{O}_{2}} \;\; \hspace*{7,2cm}\mbox {  in }\;\; Q,
\\\displaystyle -\q^{(i)}_t+yK'(t)K^{-1}(t)\nabla \q^{(i)}-\sum^{2}_{j,l,r=1}\beta_{lj}(t)
\beta_{rj}(t)\frac{\partial^2
\q^{(i)}}{{\partial y_l}{\partial y_r}}(y,t)
\displaystyle-(K^{-1}(t))^T)D\q^{(i)}\overline{\h}\\\displaystyle+\nabla {\pi}^{(i)}(y,t)K^{-1}(t)=\overline{K^{-1}(t)}\nabla\times r^{(i)}+\overline{\te}^{(i)}(K^{-1}(t))^T\cdot\nabla r+\alpha_{i}(\z-\z_{i,d})\nchi_{\mathcal{O}_{i,d}}\;\;
\mbox { in }\;\; Q,
\\\displaystyle- r^{(i)}_t+yK'(t)K^{-1}(t)\nabla r^{(i)}-\sum^{2}_{j,l,r=1}\beta_{lj}(t)
\beta_{rj}(t)\frac{\partial^2
r^{(i)}}{{\partial y_l}{\partial y_r}}(y,t)-\overline{\h}(K^{-1}(t))^T\cdot\nabla r^{(i)}
\\=\displaystyle\sum_{i=1}^2\beta_{ii}(t)\nabla\times \q^{(i)}+\sum_{i,j=1}^2(-1)^{j+1}\beta_{ij}(t)\frac{\partial \q^{(i)}_i}{\partial y_{3-j}}
+\displaystyle \widetilde{\alpha}_{i}(w-w_{i,d})\nchi_{\mathcal{O}_{i,d}} \;\; \hspace*{1,9cm}\mbox{  in } \;\; Q,
\\\nabla\cdot(M^{-1}\z^T)=0,\;\;\nabla\cdot(M^{-1}(\q^{(i)})^T)=0 \;\; \hspace*{5,9cm}\mbox{ in }\;\;Q,
\\\z=\q^{(i)}=0,\;\;w=r^{(i)}=0 \;\; \hspace*{8,1cm}\mbox{ on }\;\; \Sigma,
\\\z(0)=0,\;\;\;w(0)=0,\;\;\;\q^{(i)}(T)=0,\;\;r^{(i)}(T)=0 \;\; \hspace*{4,9cm}\mbox{ in }\;\;\Omega.\label{e2.20}
\end{array}\right.
\end{equation}

\section{On the approximate controllability}\label{Sec3}

Since we have proved the existence, uniqueness and characterization of
followers $\v^{(i)}$ and $u^{(i)}$ $(i=1,2)$, the leaders $\f$ and $g$ require now
that the  solution $(\z,w)$ of $(\ref{eq10})$, evaluated at time $t=T$, be as close
as possible to $(\z^T,w^T)$. This will be possible if the system $(\ref{eq10})$ is
approximate controllable.

The following approximate controllability result holds.

\begin{theorem}
\label{theorem 3.2} Let us consider $(\f,g)\in \textbf{L}^{2}(\mathcal{O}\times(0,T))\times L^{2}(\mathcal{O}\times(0,T))$
and $(\v(\f,g),u(\f,g))$ a Nash equilibrium in the sense (\ref{en6}), (\ref{en7}). Then the functions
\[
(\z(T),w(T))=(\z(\cdot,T,\f,g,\v,u),w(\cdot,T,\f,g,\v,u)),
\]
where $(\z,w)$ solves the system $(\ref{eq10})$, generate a dense subset of $\H \times L^{2}(\Omega)$.
\end{theorem}

\begin{proof}
By linearity of the optimality system (\ref{e2.20}), we do not restrict the problem
by assuming that
\[
\z_{i,d}=0\quad\text{and}\quad w_{i,d}=0,\quad i=1,2.
\]

To prove the density claimed, we consider $\left(  \kxi,\eta\right)  \in \H\times
L^{2}\left(  \Omega\right)$ such that
\begin{equation}
(\z(T),\kxi)_{\mathbf{L}^{2}(\Omega)}+(w(T),\eta)_{L^{2}(\Omega
)}=0,\quad\forall\,(\f,g)\in \textbf{L}^{2}(\mathcal{O}\times(0,T)) \times L^{2}(\mathcal{O}\times(0,T)).\label{e3.19}%
\end{equation}
We want to show that $\left(  \kxi,\eta\right)  =0.$ For this, let us consider the
coupled system:
\begin{equation}
\left\{
\begin{array}
[c]{lll}%
\displaystyle -\vph_t+yK'(t)K^{-1}(t)\nabla \vph-\sum^{2}_{j,l,r=1}\beta_{lj}(t)
\beta_{rj}(t)\frac{\partial^2
\vph}{{\partial y_l}{\partial y_r}}(y,t) - (K^{-1}(t))^T (D\vph)\h
\\\displaystyle+\nabla \pi K^{-1}(t)=\overline{K^{-1}(t)}\nabla\times \psi  +  \overline{K^{-1}(t)}\te\nabla\psi  +\alpha_{1}\be^{(1)}\chi_{\mathcal{O}_{1,d}}+\alpha
_{2}\be^{(2)}\chi_{\mathcal{O}_{2,d}} & \mathrm{in} & Q,\\
\displaystyle -\psi_t+yK'(t)K^{-1}(t)\nabla \psi-\sum^{2}_{j,l,r=1}\beta_{lj}(t)
\beta_{rj}(t)\frac{\partial^2
\psi}{{\partial y_l}{\partial y_r}}(y,t)-\h(K^{-1}(t))^T\cdot\nabla \psi
\\=\displaystyle\sum_{i=1}^2\beta_{ii}\nabla\times \vph-\sum_{i,j=1}^2(-1)^{j+1}\beta_{ij}\frac{\partial \vph_i}{\partial y_{3-j}} +\widetilde
{\alpha}_{1}\gamma^{(1)}\chi_{\mathcal{O}_{1,d}}+\widetilde{\alpha}_{2}%
\gamma^{(2)}\chi_{\mathcal{O}_{2,d}} & \mathrm{in} & Q,\\
\displaystyle \be_t^{(i)}-yK'(t)K^{-1}(t)\nabla \be^{(i)}-\sum^{2}_{j,l,r=1}\beta_{lj}(t)
\beta_{rj}(t)\frac{\partial^2
\be^{(i)}}{{\partial y_l}{\partial y_r}}(y,t)
+\displaystyle (\overline{h}(K^{-1}(t)^T.\nabla))\be^{(i)} \\
+ (\be^{(i)}(K^{-1}(t)^T)\cdot \nabla)\overline{h} + \nabla\overline{p}^{(i)}(y,t)K^{-1}(t) - \overline{K^{-1}(t)}\nabla\times \gamma^{(i)} = -\dfrac{1}{\mu}_{i}\vph\chi_{\mathcal{O}_{i}} & \mathrm{in} & Q,\\
\displaystyle \gamma_t^{(i)}-yK'(t)K^{-1}(t)\nabla \gamma^{(i)}-\sum^{2}_{j,l,r=1}\beta_{lj}(t)
\beta_{rj}(t)\frac{\partial^2
\gamma^{(i)}}{{\partial y_l}{\partial y_r}}(y,t)+\overline{\h}(K^{-1}(t))^T\cdot\nabla \gamma^{(i)}\\
\displaystyle +\be^{(i)}(K^{-1}(t))^T\cdot\nabla\overline{\te}- \sum_{i=1}^2\beta_{ii}(t)\nabla\times \be^{(i)}-\sum_{i,j=1}^2(-1)^{j+1}\beta_{ij}(t)\frac{\partial \be^{(i)}_i}{\partial y_{3-j}}= -\frac{1}{\tilde{\mu_i}}\psi\chi_{\mathcal{O}_i}& \mathrm{in} & Q,\\
\nabla\cdot (M^{-1}\vph^{T})=0,\quad\nabla\cdot (M^{-1}(\be^{(i)})^T)=0 & \mathrm{in} & Q,\\
\vph=\be^{(i)}=0,\quad \psi=\gamma^{(i)}=0 & \mathrm{on} & \Sigma,\\
\vph(T)=\kxi,\quad \psi(T)=\eta,\quad\be^{(i)}(0)=0,\quad\gamma^{(i)}(0)=0 &
\mathrm{in} & \Omega.
\end{array}
\right.  \label{e3.3}%
\end{equation}
\\
Multiplying $(\ref{e3.3})_1$ by $\z$, $(\ref{e3.3})_2$ by $w$,  integrating in $Q$,
making integrations by parts and using  (\ref{e2.20}), it follows that
\begin{equation}\label{mat10}
\begin{array}{rcl}
(\vph, \overline{K^{-1}(t)}\nabla\times w+\f\nchi_{\mathcal{O}} -\displaystyle\frac{1}{\mu_1}\q^{(1)}\nchi_{\mathcal{O}_1}-\frac{1}{\mu_2}\q^{(2)}\nchi_{\mathcal{O}_2})_{\textbf{L}^{2}(Q)} - (\kxi, \z(T))_{\mathbf{L}^{2}(\Omega)} \\[10pt]
\hspace{3 cm} =   ( \overline{K^{-1}(t)}\nabla\times \psi  +  \overline{K^{-1}(t)}\te\nabla\psi  +\alpha_{1}\be^{(1)}\chi_{\mathcal{O}_{1,d}}+\alpha
_{2}\be^{(2)}\chi_{\mathcal{O}_{2,d}},\z)_{\textbf{L}^{2}(Q)}
\end{array}
\end{equation}
and
\begin{align}\label{mat11}
\nonumber &(\psi, - \z(K^{-1}(t))^T\cdot\nabla\te +  \displaystyle\sum_{i=1}^2\beta_{ii}\nabla\times \z+\sum_{i,j=1}^2(-1)^{j+1}\beta_{ij}\frac{\partial \z_i}{\partial y_{3-j}}
+g\chi_{\mathcal{O}}\displaystyle -\dfrac{1}{\widetilde{\mu}_{1}}r^{(1)}\chi_{\mathcal{O}_{1}}
-\dfrac{1}{\widetilde{\mu}_{2}}r^{(2)}\chi_{\mathcal{O}_{2}})_{L^{2}(Q)}\\[10pt]
&-  (\eta, w(T))_{L^{2}(\Omega)}=  ( \displaystyle\sum_{i=1}^2\beta_{ii}\nabla\times \vph-\sum_{i,j=1}^2(-1)^{j+1}\beta_{ij}\frac{\partial \vph_i}{\partial y_{3-j}} +\widetilde
{\alpha}_{1}\gamma^{(1)}\chi_{\mathcal{O}_{1,d}} +\widetilde{\alpha}_{2}%
\gamma^{(2)}\chi_{\mathcal{O}_{2,d}}, w )_{L^{2}(Q)}
\end{align}

Adding both members of (\ref{mat10}) and (\ref{mat11}), we deduce that
\begin{equation}\label{e3.20}
\begin{array}{rcl}
(\vph,\f\chi_{\mathcal{O}}-\dfrac{1}{\mu_{1}}{\q}^{(1)}\chi_{\mathcal{O}_{1}%
}-\dfrac{1}{\mu_{2}}{\q}^{(2)}\chi_{\mathcal{O}_{2}})_{\textbf{L}^{2}(Q)}%
 + (\psi,g\chi_{\mathcal{O}}-\dfrac{1}{\widetilde{\mu}_{1}}r^{(1)}\chi_{\mathcal{O}_{1}}%
-\dfrac{1}{\widetilde{\mu}_{2}}r^{(2)}\chi_{\mathcal{O}_{2}})_{L^{2}(Q)}\\[10pt]
\disp - (\kxi,{\z}(T))_{\textbf{L}^{2}(\Omega)}- (\eta,w(T))_{L^{2}(\Omega)} =(\alpha_{1}\be^{(1)}\chi_{\mathcal{O}_{1,d}}+\alpha_{2}\be^{(2)}\chi_{\mathcal{O}_{2,d}},\z)_{\textbf{L}^{2}(Q)}\\
\disp +  (\widetilde{\alpha}_{1}\gamma^{(1)}\chi_{\mathcal{O}_{1,d}}+
\widetilde{\alpha}_{2}\gamma^{(2)}\chi_{\mathcal{O}_{2,d}},w)_{L^{2}(Q)}%
\end{array}
\end{equation}

Analogously, multiplying  $(\ref{e3.3})_3$ by $\q$ , $(\ref{e3.3})_4$ by $r$,
integrating in $Q$, making integrations by parts, using  (\ref{e2.20})  (with $\left(
z_{i,d},w_{i,d}\right) =0,~i=1,2$) and summing the resulting equations, we obtain
\begin{equation}\label{e3.21}%
\begin{array}{l}
-(\dfrac{1}{\mu_{i}}\vph\chi_{\mathcal{O}_{i}},\q^{(i)})_{\textbf{L}^{2}(Q)}-   (\dfrac{1}{\widetilde{\mu}_{i}}\psi\chi_{\mathcal{O}_{i}},r^{(i)})_{L^{2}(Q)} =\alpha
_{i}(\be^{(i)},\z)_{\textbf{L}^{2}(\mathcal{O}_{i,d}\times(0,T))}\\[10pt]
 \hspace {7 cm} + \widetilde{\alpha}_{i}(\gamma^{(i)},w)_{L^{2}(\mathcal{O}_{i,d
}\times(0,T)},\quad i=1,2.
\end{array}
\end{equation}
Substituting (\ref{e3.21}) in (\ref{e3.20}) it follows that
\begin{equation} \label{e3.23}%
(\z(T),\kxi)_{\textbf{L}^{2}(\Omega)} + (w(T),\eta)_{L^{2}(\Omega)} = (\vph,\f\chi_{\mathcal{O}})_{\textbf{L}^{2}(Q)} + (\psi,g\chi_{\mathcal{O}}%
)_{L^{2}(Q)},
\end{equation}
for all $(\f, g)  \in  \textbf{L}^{2}(\mathcal{O} \times(0,T)) \times
L^{2}(\mathcal{O}\times(0,T)).$

From the above identity and the condition (\ref{e3.19}), we have
\begin{equation}
(\vph,\psi)\equiv0\quad\mathrm{in\quad\mathcal{O}\times(0,T)}.\label{e3.24}%
\end{equation}

It follows from  (\ref{e3.24}), together with the hypothesis
$\mathcal{O}_{i}\subset\mathcal{O}$ ($i=1,2$), that the right hand-side of the
system for $(\be^{(i)},\gamma^{(i)},\bar{p}^{(i)})$ in (\ref{e3.3}) vanish. Therefore
\begin{equation}
(\be^{(i)},\gamma^{(i)})\equiv 0\quad\mathrm{in}\quad Q.\label{en13}
\end{equation}
because $(\be^{(i)}(0),\gamma^{(i)}(0))=0$. Substituting (\ref{en13}) in (\ref{e3.3}),
we obtain that $(\vph,\psi)$ is solution of the following system:
\begin{equation}
\left\{
\begin{array}
[c]{lll}%
\displaystyle -\vph_t+yK'(t)K^{-1}(t)\nabla \vph-\sum^{2}_{j,l,r=1}\beta_{lj}(t)
\beta_{rj}(t)\frac{\partial^2
\vph}{{\partial y_l}{\partial y_r}}(y,t) -(K^{-1}(t))^T (D\vph)\h
\\\displaystyle+\nabla \pi K^{-1}(t)=\overline{K^{-1}(t)}\nabla\times \psi  +  \overline{K^{-1}(t)}\te\nabla\psi &\mathrm{in} & Q,\\
\displaystyle -\psi_t+yK'(t)K^{-1}(t)\nabla \psi-\sum^{2}_{j,l,r=1}\beta_{lj}(t)
\beta_{rj}(t)\frac{\partial^2
\psi}{{\partial y_l}{\partial y_r}}(y,t)-\h(K^{-1}(t))^T\cdot\nabla \psi
\\=\displaystyle\sum_{i=1}^2\beta_{ii}\nabla\times \vph-\sum_{i,j=1}^2(-1)^{j+1}\beta_{ij}\frac{\partial \vph_i}{\partial y_{3-j}}
&\mathrm{in} & Q,\\
\nabla\cdot (M^{-1}\vph^{T})=0  & \mathrm{in} & Q,\\
\vph=0,\quad \psi=0  &\mathrm{in}& \mathcal{O} \times (0,T),\\
\vph=0,\quad \psi=0 &\mathrm{on}& \Sigma,\\
\vph(T)=\kxi,\quad \psi(T)=\eta &\mathrm{in} & \Omega.
\end{array}
\right. \label{e3.25}%
\end{equation}

From (\ref{e3.24}), (\ref{e3.25}) and the Calerman estimate proved by
Fernandez-Cara and Guerrero (cf.  \cite{FCG}, Proposition 1), we obtain
\begin{equation} \label{eq60}
(\vph,\psi)\equiv0\quad\mathrm{in}\quad Q.
\end{equation}

In fact, the coefficient of the principal part
$$\disp - \sum^{2}_{j,l,r=1}\beta_{lj}(t)\beta_{rj}(t)\frac{\partial^2 \vph}{{\partial y_l}{\partial y_r}}(y,t),\;\;  -\sum^{2}_{j,l,r=1}\beta_{lj}(t)\beta_{rj}(t)\frac{\partial^2\psi}{{\partial y_l}{\partial y_r}}(y,t)$$
according to the assumption on $\disp K(t)$ are of class $C^2$.

In particular $(\kxi,\eta)\equiv0$. This complete the proof.
\end{proof}

\section{Optimality system for the leader}\label{Sec4}

Thanks to the results obtained in Section \ref{Sec2}, we can take, for  each pair
$(\f,g)$, the Nash equilibrium $(\v,w)$ associated to solution $(\z,w)$ of
(\ref{eq10}). We will show the existence of a  pair of leader control
$(\overline{\f},\overline{g})$ solution of the following problem:
\begin{equation} \label{eq61}
\inf_{(\f,g)\, \in\, \mathcal{U}_{ad}} \frac{1}{2}\int_{0}^{T} \int_{\mathcal{O}}|\det K(t)|\big(|\f|^2 + |g|^2\big)dydt,
\end{equation}
where $\mathcal{U}_{ad}$ is the set of admissible controls
\begin{equation*}
\mathcal{U}_{ad} = \{(\f,g) \in \mathbf{L}^2\big(\mathcal{O} \times (0,T)\big) \times L^2\big(\mathcal{O} \times (0,T)\big); (\z,w) \ \text{solution of (\ref{eq10}) satisfying (\ref{en9})} \}.
\end{equation*}

For this, we will use  a duality argument due to Fenchel and Rockfellar \cite{R}
(cf. also  \cite{Bre}, \cite{EK}). Before, we recall the following definition from
convex analysis:

\begin{definition}\label{def123} Let E be a real linear locally convex space and  E' be its dual space. Consider any function $G: E \rightarrow ( -\infty, +\infty]$. The
function $G^{*}:  E'  \rightarrow ( -\infty, +\infty]$ defined by
$$G^{*}(x') = \sup_{x\,  \in\,  E}\{\langle x', x\rangle - G(x) \}, \;\; x' \in  E'$$
is called the conjugate function of G (sometimes called the Legendre transform
of G).
\end{definition}

The following result holds:
\begin{theorem} \label{theor3}
The optimal leader control, that  is, the leader control which to solve the problem
(\ref{eq61}), is given by $(\overline{\f},\overline{g}) = (\vph \chi_{\mathcal{O}},\psi
\chi_{\mathcal{O}})$, where $(\vph,\psi,\pi, \be^{(i)},\gamma^{(i)},\bar{p}^{(i)})$ is
the  solution of (\ref{e3.3}) associate to the  pair $\left( \kxi,\eta\right)  \in
\mathbf{L}^2(\Omega) \times L^{2}\left( \Omega\right)$ which solves the dual
minimization problem
$$\inf_{(\kxi, \eta)}\Bigg\{\frac{1}{2}\int_{\mathcal{O} \times (0,T)}|\det K(t)|\big(|\vph|^2 + |\psi|^2\big)dydt + \epsilon||(\kxi,\eta)||_{\mathbf{L}^2(\Omega) \times L^2(\Omega)}
 - \big((\kxi,\eta),(\z^T,w^T)\big)_{\mathbf{L}^2(\Omega) \times L^2(\Omega)}\Bigg\}$$
\end{theorem}

\begin{proof} We define the continuous linear operator $L : \mathbf{L}^2\big(\mathcal{O}
\times (0,T) \times L^2\big(\mathcal{O} \times (0,T) \longrightarrow
\mathbf{L}^2(\Omega) \times L^2(\Omega)$  by
\begin{equation*}
L(\f,g) = \big(\z(\cdot,T;\f,g,\v,w),w(\cdot,T;\f,g,\v,w)\big)
\end{equation*}
and we introduce
\begin{equation*}
F_1(\f,g) = \frac{1}{2} \int_{\mathcal{O} \times (0,T)}|\det K(t)|\big(|\f|^2 + |g|^2\big)dydt
\end{equation*}
and
\begin{equation*}
F_2(\kxi,\eta) =
\left\{
\begin{array}{l}
\displaystyle 0, \ \text{if} \ \big(\kxi + \z(T), \eta + w(T)\big) \in B_{\mathbf{L}^2}(\z^T,\epsilon) \times B_{L^2}(w^T,\epsilon),\\
\displaystyle \infty, \ \text{otherwise}.
\end{array}
\right.
\end{equation*}
Then, the problem (\ref{eq61}) becomes equivalent to
\begin{equation} \label{eq62}
\inf_{(\f,g)}\big[F_1(\f,g) + F_2(L(\f,g))\big],
\end{equation}
and, by the Duality Theorem of Fenchel and Rockfellar \cite{R} ( see also
\cite{Bre}, \cite{EK}), we have
\begin{equation} \label{eq63}
\inf_{(\f,g)}\big[F_1(\f,g) + F_2(L(\f,g))\big] = \inf_{(\kxi,\eta)}\big[F_{1}^{*}(L^*(\kxi,\eta)) + F_{2}^{*}(-(\kxi,\eta))\big]
\end{equation}
where $L^* : \mathbf{L}^2(\Omega) \times L^2(\Omega) \longrightarrow
\mathbf{L}^2\big(\mathcal{O} \times (0,T)\big) \times L^2\big(\mathcal{O} \times
(0,T)\big)$ is the adjoint of $L$ and $F_{i}^{*}$ is the conjugate function of $F_{i} (
i = 1,2)$(see  Definition (\ref{def123})).

From (\ref{e3.3}) we obtain (\ref{e3.23}). In this way, we have
\begin{align*}
\big(L^*(\kxi,\eta),(\f,g)\big)_{\mathbf{L}^2(\mathcal{O} \times (0,T)) \times L^2(\mathcal{O} \times (0,T))} &= \big((\kxi,\eta),L(\f,g)\big)_{\mathbf{L}^2(\Omega) \times L^2(\Omega)}\\
& = \big((\vph,\psi),(\f,g)\big)_{\mathbf{L}^2(\mathcal{O} \times (0,T)) \times L^2(\mathcal{O} \times (0,T))}
\end{align*}
and, therefore
\begin{equation} \label{eq64}
L^*(\kxi,\eta) = (\vph,\psi)
\end{equation}
Notice that
\begin{equation}\label{eq641}
F_{1}^{*}(\vph,\psi) = \frac{1}{2}\int_{\mathcal{O} \times (0,T)}|\det K(t)|\big(|\vph|^2 + |\psi|^2\big)dydt
\end{equation}
and
\begin{equation}\label{eq642}
F_{2}^{*}\big(-(\kxi,\eta)\big) = \epsilon||(\kxi,\eta)||_{\mathbf{L}^2(\Omega) \times L^2(\Omega)} - \big((\kxi,\eta),(\z^T,w^T)\big)_{\mathbf{L}^2(\Omega) \times L^2(\Omega)}.
\end{equation}
So, from  (\ref{eq64})-- (\ref{eq642}), (\ref{eq63}) becomes
\begin{equation} \label{eq65}
\inf_{(\f,g)}\big[F_1(\f,g) + F_2(L(\f,g))\big] = -\inf_{(\kxi,\eta)}\Theta(\kxi,\eta),
\end{equation}
where the functional $\Theta: \mathbf{L}^2(\Omega) \times L^2(\Omega) \longrightarrow \Bbb R$ is defined by
\begin{align*}
\Theta(\kxi,\eta) &= \frac{1}{2}\int_{\mathcal{O} \times (0,T)}|\det K(t)|\big(|\vph|^2 + |\psi|^2\big)dydt + \epsilon||(\kxi,\eta)||_{\mathbf{L}^2(\Omega) \times L^2(\Omega)}\\
& - \big((\kxi,\eta),(\z^T,w^T)\big)_{\mathbf{L}^2(\Omega) \times L^2(\Omega)}
\end{align*}

Since the functional $\Theta$ is continuous and strictly convex, it follows by a
result of convex analysis(see, for instance, \cite{Bre}) that to  solve the problem
(\ref{eq65}) and consequently (\ref{eq61}), it is sufficient to prove that the
functional $\Theta$ is coercive. More precisely,
\begin{equation} \label{eq66}
\liminf_{||(\kxi,\eta)||_{\mathbf{L}^2(\Omega) \times L^2(\Omega)} \rightarrow \infty} \frac{\Theta(\kxi,\eta)}{||(\kxi,\eta)||_{\mathbf{L}^2(\Omega) \times L^2(\Omega)}} \geq \epsilon.
\end{equation}
In order, to prove (\ref{eq66}), we consider a sequence $(\kxi_n,\eta_n) \in \mathbf{L}^2(\Omega) \times L^2(\Omega)$ such that
\begin{equation} \label{eq67}
||(\kxi_n,\eta_n)||_{\mathbf{L}^2(\Omega) \times L^2(\Omega)} \rightarrow \infty, \ \ \text{as} \ \ n \rightarrow \infty
\end{equation}

We denote by $(\vph_n,\psi_n,\be_{n}^{(i)},\gamma_{n}^{(i)})$ the corresponding
sequence of solution of (\ref{e3.3}).

For $(\widehat{\kxi}_n,\widehat{\eta}_n) = (\kxi_n,\eta_n)/||(\kxi_n,\eta_n)||_{\mathbf{L}^2(\Omega) \times L^2(\Omega)}$ we write
\begin{equation*}
(\widehat{\vph}_n,\widehat{\psi}_n,\widehat{\be}_{n}^{(i)},\widehat{\gamma}_{n}^{(i)}) = \left(\frac{\vph_n}{||(\kxi_n,\eta_n)||_{\mathbf{L}^2(\Omega) \times L^2(\Omega)}},\frac{\psi_n}{||(\kxi_n,\eta_n)||_{\mathbf{L}^2(\Omega) \times L^2(\Omega)}},\be_{n}^{(i)},\gamma_{n}^{(i)}\right)
\end{equation*}
the solution of (\ref{e3.3}) with $\big(\vph_n(T),\psi_n(T)\big) = (\widehat{\kxi}_n,\widehat{\eta}_n)$. Thus
\begin{align}
\nonumber \frac{\Theta(\kxi_n,\eta_n)}{||(\kxi_n,\eta_n)||_{\mathbf{L}^2(\Omega) \times L^2(\Omega)}} &= \frac{||(\kxi_n,\eta_n)||_{\mathbf{L}^2(\Omega) \times L^2(\Omega)}}{2}\int_{\mathcal{O} \times (0,T)}|\det K(t)|\big(|\widehat{\vph}_n|^2 + |\widehat{\psi}_n|^2\big)dydt \\
\nonumber & + \epsilon||(\widehat{\kxi}_n,\widehat{\eta}_n)||_{\mathbf{L}^2(\Omega) \times L^2(\Omega)} - \big((\widehat{\kxi}_n,\widehat{\eta}_n),(\z^T,w^T)\big)_{\mathbf{L}^2(\Omega) \times L^2(\Omega)}\\
& \geq \frac{||(\kxi_n,\eta_n)||_{\mathbf{L}^2(\Omega) \times L^2(\Omega)}}{2}\int_{\mathcal{O} \times (0,T)}|\det K(t)|\big(|\widehat{\vph}_n|^2 + |\widehat{\psi}_n|^2\big)dydt \label{eq68}\\
\nonumber & + \epsilon - \big((\widehat{\kxi}_n,\widehat{\eta}_n),(\z^T,w^T)\big)_{\mathbf{L}^2(\Omega) \times L^2(\Omega)}
\end{align}

We distinguish the following two cases:
\begin{itemize}
\item[\textbf{Case 1:}] $\displaystyle \liminf_{n \rightarrow \infty} \int_{\mathcal{O} \times (0,T)}|\det K(t)|\big(|\widehat{\vph}_n|^2 + |\widehat{\psi}_n|^2\big)dydt > 0$;

\item[\textbf{Case 2:}] $\displaystyle \liminf_{n \rightarrow \infty} \int_{\mathcal{O} \times (0,T)}|\det K(t)|\big(|\widehat{\vph}_n|^2 + |\widehat{\psi}_n|^2\big)dydt = 0$.
\end{itemize}

In the first case, due to (\ref{eq67}), we get by (\ref{eq68}) that
\begin{equation*}
\liminf_{||(\kxi_n,\eta_n)||_{\mathbf{L}^2(\Omega) \times L^2(\Omega)} \rightarrow \infty}\frac{\Theta(\kxi_n,\eta_n)}{||(\kxi_n,\eta_n)||_{\mathbf{L}^2(\Omega) \times L^2(\Omega)}} = +\infty.
\end{equation*}

Let us now analyze the second case. Let us consider a subsequence (still denoted by the index $n$ to simplify the notation) such that, as $n \rightarrow \infty$,
\begin{equation} \label{eq69}
\int_{\mathcal{O} \times (0,T)}|\det K(t)|\big(|\widehat{\vph}_n|^2 + |\widehat{\psi}_n|^2\big)dydt \rightarrow 0
\end{equation}
and
\begin{equation} \label{eq70}
(\widehat{\kxi}_n,\widehat{\eta}_n) \rightarrow (\widehat{\kxi},\widehat{\eta}) \ \text{weakly in} \ \mathbf{L}^2(\Omega) \times L^2(\Omega).
\end{equation}

Consequently
\begin{equation*}
(\widehat{\vph}_n,\widehat{\psi}_n,\widehat{\be}_{n}^{(i)},\widehat{\gamma}_{n}^{(i)}) \rightarrow (\widehat{\vph},\widehat{\psi},\widehat{\be}^{(i)},\widehat{\gamma}^{(i)}) \ \text{weakly} \ L^2\big(0,T;(\mathbf{V} \times H_{0}^{1}(\Omega))^2\big),
\end{equation*}
where $(\widehat{\vph},\widehat{\psi},\widehat{\be}^{(i)},\widehat{\gamma}^{(i)})$ is the solution of (\ref{e3.3}) with $\big(\widehat{\vph}(T), \widehat{\psi}(T)\big) = \big(\widehat{\kxi},\widehat{\eta}\big)$.\\
According to (\ref{eq69}) we deduce that
\begin{equation*}
(\widehat{\vph},\widehat{\psi}) \equiv 0 \ \ \text{in} \ \ \mathcal{O} \times (0,T).
\end{equation*}

In this way, following the same arguments to obtain (\ref{eq60}), we have
\begin{equation*}
(\widehat{\vph},\widehat{\psi}) \equiv 0 \ \text{in} \ Q
\end{equation*}
and, therefore, $(\widehat{\kxi},\widehat{\eta}) \equiv 0$. Thus, from (\ref{eq70}) we deduce
\begin{equation} \label{eq71}
(\widehat{\kxi}_n,\widehat{\eta}_n) \rightarrow 0 \ \text{weakly in} \ \mathbf{L}^2(\Omega) \times L^2(\Omega).
\end{equation}
As a consequence of (\ref{eq71}), we can take the limit in (\ref{eq68}) to obtain
(\ref{eq66}). This completes the proof. \end{proof}

\section{Conclusions}\label{sec5}

The main achievements obtained in this article were  the existence and
uniqueness of Nash equilibrium and its characterization, the approximate
controllability of the linearized micropolar system with respect to the leader
control  and the existence and uniqueness  of the Stackelberg-Nash problem,
where the optimality system for the leader is given.

The results obtained  can generate several interesting problems, generalizing or
improving  results obtained  in other models, for instance, the problem of
extending the results obtained  here  for  the nonlinear micropolar fluids system.

In the article \cite{RA2}   it is called the attention for the Nash equilibrium
problem for systems governed by  nonlinear partial differential equations, in
particular for the Navier-Stokes and Burgers equations. For the Burgers equation,
the  problem was solved, but for the Navier-Stokes ones the problem is quite
difficult from both  mathematical and  numerical point of view. To study Nash
equilibria associated to  Navier-Stokes equations and their variants  are
interesting open problems.

We also mention  the work by Ramos and Roubicek \cite{RR}, where the authors
studied a  nonlinear PDE model from a theoretical and  numerical point of view.
\vspace{1cm}

{\bf Acknowledgements}  The author wants to express his gratitude to the
anonymous reviewers for their questions and commentaries; they were very
helpful in improving this article. The author also thanks Newton Santos for his
comments on the manuscript.

\end{document}